\documentclass{article}[12pt]
\usepackage[a4paper, total={6.5in, 9in}]{geometry}
\usepackage{amssymb}
\usepackage{amsthm}
\usepackage{amsmath}
\usepackage{enumitem}
\usepackage{hyperref}
\usepackage{authblk}
\usepackage{tikz,pgfplots}
\usetikzlibrary{shapes}
\usepackage{mathtools}

\newtheorem{theorem}{\large Theorem}[section]
\newtheorem{lemma}{\large Lemma}[section]
\newtheorem{corollary}{\large Corollary}[section]
\newtheorem{definition}{\large Definition}[section]
\newtheorem{remark}{\large Remark}[section]
\newtheorem{example}{\large Example}[section]

\def\1{\rule{0pt}{1.7ex}pq}
\def\a{\rule{0pt}{1.1ex}pq}

\def\2{\rule{0pt}{1.7ex}p_0p}
\def\3{\rule{0pt}{2ex}X_p}
\def\x{\rule{0pt}{1.7ex}p}
\def\4{\rule{0pt}{1.7ex}1}
\def\5{\rule{0pt}{1.7ex}2}
\def\6{\rule{0pt}{1.7ex}p_0q}
\def\7{\rule{0pt}{1.7ex}p_0p_1}
\def\8{\rule{0pt}{1.7ex}p_0p_2}

\begin{document}

	\thispagestyle{empty}
	
	\title{Optimality Conditions for Interval-Valued Optimization Problems on Riemannian Manifolds Under a Total Order Relation}
	
	\author[a]{Hilal Ahmad Bhat}

	\author[a,*]{Akhlad Iqbal}

    \author[a]{Mahwash Aftab}

\affil[a]{\it \small Department of Mathematics, Aligarh Muslim University, Aligarh, 202002, Uttar Pradesh, India}

%
\date{}

\maketitle

\let\thefootnote\relax\footnotetext{*corresponding author (Akhlad Iqbal)\\ Email addresses: bhathilal01@gmail.com (Hilal Ahmad Bhat),\\ akhlad6star@gmail.com (Akhlad Iqbal),\\ mahwashaftab02@gmail.com (Mahwash Aftab)}


\begin{abstract}
	\noindent This article explores fundamental properties of convex interval-valued functions defined on Riemannian manifolds. The study employs generalized Hukuhara directional differentiability to derive KKT-type optimality conditions for an interval-valued optimization problem on Riemannian manifolds. Based on type of functions involved in optimization problems, we consider the following cases:
	\begin{itemize}
		\item [\textbullet] objective function as well as constraints are real-valued;
		\item[\textbullet]  objective function is interval-valued and constraints are real-valued;
		\item[\textbullet] objective function as well as contraints are interval-valued.
	\end{itemize} 
The whole theory is justified with the help of examples. The order relation that we use throughout the paper is a total order relation defined on the collection of all closed and bounded intervals in $\mathbb{R}$.
\end{abstract}

\section{Introduction}
Randomness, inexactness and imprecisions are natural to occur in real life decision making problems. An efficient decision making under uncertain environment leads to imposition of uncertainty in classical optimization programming problems. To tackle uncertainty in optimization programming problems, researchers have developed many optimization techniques and methods which are broadly classified into three different classes such as stochastic optimization programming (SOP), fuzzy optimization programming (FOP) and interval-valued optimization programming (IVOP). SOP and FOP respectively include use of random variables and fuzzy numbers which are subjective in nature, and it is hard to relate such methods to real life problems. The IVOP on the other hand provides an easier approach to tackle uncertainty in optimization programming problems. In IVOP, a closed and bounded interval in $\mathbb{R}$ is used to represent the uncertainty of a variable. Moreover, the coefficients of functions involved in an IVOP are closed and bounded intervals in $\mathbb{R}$.

In order to solve IVOP problems, many methods have been explored by various researchers. A basic overview of interval analysis is given by Moore \cite{moore2,moore1}, and Alefeld and Herzberger \cite{alefeld}. Ishibuchi and Tanaka \cite{ishibuchi} introduced the ordering relation of two closed and bounded intervals in terms of their center and half-width (radius) and derived the solution concepts for multi-objective IVOP problems. For the KKT optimality conditions of an IVOP and multi-objective IVOP problem defined on Euclidean spaces, one can refer to \cite{chalco-cano,sadikur,bilal1,bilal2,wu1,wu2}. 

Furthermore, several authors have laid focus on the extension of the methods and techniques developed for solving IVOP problems on Euclidean spaces to Riemannian manifolds, see \cite{bento2,bento1,ferreira,gabriel,li,nemeth,rapcsak1,rapcsak2,udriste,wang}. Such extensions have their own benefits such as, a non-convex optimization program defined on Euclidean space can turnout to be a convex program when introduced to a Riemannian manifold under a suitable Riemannian metric \cite{chen,rapcsak1,rapcsak2}. Moreover, a non-monotone vector field turns out to be a monotone vector field when extended to a suitable Riemannian manifold \cite{ferreira,rapcsak1,rapcsak2}. Udriste \cite{udriste} and Rapcsak \cite{rapcsak2} are the first authors who considered such extensions. Chen \cite{chen} presented the sufficient KKT optimality conditions for a convex IVOP problem on a Hadamard manifold under a partial order relation.

However, the order relation defined by Ishibuchi and Tanaka \cite{ishibuchi} is not complete in the sense that any two closed and bounded intervals in $\mathbb{R}$ are not comparable. Utilizing a total order relation introduced by Bhunia and Samanta \cite{bhunia}, we study few fundamental properties of interval-valued convex functions on Riemannian manifolds. We employ the gH-directional derivative to study KKT type optimality (sufficient) conditions of an IVOP problem on Riemannian manifolds. We present an example of an IVOP problem on which the KKT conditions developed on Euclidean spaces by the authors \cite{chalco-cano,chen,sadikur,bilal1,bilal2,wu1,wu2} can't be applied, however, the same problem can be solved by the techniques developed in this paper.

\section{Preliminaries}
In this section, we revisit fundamental definitions, notations, and established findings concerning Riemannian manifolds that will be employed consistently in the entirety of this article. For more details see \cite{rapcsak2,sakai,loring,udriste}.

Let $(M,g)$ be a complete finite dimensional Riemannian manifold with Riemannian metric $g$ and Riemannian connection $\nabla$ on $M$. The tangent space of $M$ at $p_0$ is denoted by $T_{p_0}(M)$ and the tangent space of a subset $E \subseteq (M,g)$ at $p_0\in E$ is denoted by $T_{p_0}(E)$. Given a piece-wise smooth curve $\gamma_{\1} : [a_1,a_2] \rightarrow M$ joining $p$ and $q$ i.e., $\gamma_{\1}(a_1)=p$ and $\gamma_{\1}(a_2) =q$, the length of $\gamma_{\1}$ is given by $ L(\gamma_{\1}(s))= {\displaystyle\int_a^b} \left.\sqrt{g(\gamma_{\1}'(s),\gamma_{\1}'(s))}\right|_{\gamma_{\a}(s)} ds$ and the Riemannian distance between $p$ and $q$ is given by $d(p,q)=\underset{\gamma_{\a}}{\text{inf}}~ L(\gamma_{\1})$. A vector field on $M$ is a mapping of $M$ into $TM~\left(=\underset{p\in M}{\bigcup}T_p(M)\right)$ which associates to each $p\in M$ a vector $X_p\in T_p(M)$. A vector field $X$ is said to be parallel along $\gamma$ if $\nabla_{\gamma'}X=0$. We say $\gamma$ is a geodesic if $\gamma'$ is itself parallel along $\gamma$. A geodesic $\gamma_{\1}$ joining $p$ and $q$ is minimal if $L(\gamma)=d(p,q)$. By Holf-Rinow theorem, we have 
\begin{itemize}
	\item[\textbullet] $(M,d)$ is a complete metric space;
	\item[\textbullet] closed and bounded subsets of $M$ are compact;
	\item [\textbullet] any two points in $M$ can be joined by a minimal geodesic.
\end{itemize}

\noindent For any $X_p\in T_p(M)$, the exponential map $exp_{\x}:T_p(M) \rightarrow M$ at $p$ is defined by $exp_{\x}(X_p)=\gamma_{\3}(1)$, where $\gamma_{\3}(s),~s\in I,~ 0 \in I,$ is a geodesic emanating from $p=\gamma_{\3}(0)$ in the direction $X_p=\dot{\gamma}(0)$ and $I$ is an interval in $\mathbb{R}$. It satisfies the following property
$$exp_{\x} (s X_p)=\gamma_{s \3}(1)=\gamma_{\3}(s).$$
The exponential map $exp_{\x}$ is differentiable at $p$ and its differential is an identity map.

\noindent Now we recall the basic arithmetics of intervals.

We denote by $\mathbb{I}$ the collection of closed and bounded intervals of $\mathbb{R}$. Let $T\in \mathbb{I}$, we write $T=[t^l,t^u]$ where $t^l$ and $t^u$ are lower and upper bounds of $T$, respectively. For $T_1,T_2 \in \mathbb{I}$ and $n \in \mathbb{R}$, we have
$$T_1+T_2=\{t_1+t_2: t_1\in T_1,t_2 \in T_2\}=[t_1^l+t_2^l,t_1^u+t_2^u]$$
$$nT_1=\{nt_1:t_1\in T_1\}=\begin{cases}
	[nt_1^l,nt_1^u], & n \geq 0;\\
	[nt_1^u,nt_1^l], & n<0.
\end{cases}$$
From the above two expressions, one has
$$-T_1=[-t_1^u,-t_1^l] ~~ \text{and}~~ T_1-T_2=[t_1^l-t_2^u,t_1^u-t_2^l].$$
The Hausdorff distance between $T_1$ and $T_2$ is 
\begin{equation} \label{metric,p2}
	d_H(T_1,T_2) =\text{max}\{|t_1^l-t_2^l|,|t_1^u-t_2^u|\}.
\end{equation}
For more details, we refer to \cite{alefeld,moore2,wu1}.

We can also represent an interval $T\in \mathbb{I}$ in terms of its center and half width (radius) as
\begin{equation} \label{eq1,p2}
	T=\langle t^c,t^w \rangle,
\end{equation}
where $t^c=\frac{t^l+t^u}{2}$ and $t^w=\frac{t^u-t^l}{2}$ are respectively the center and half-width of $T$. Throughout the paper, we will consider the representation (\ref{eq1,p2}) of an interval $T\in \mathbb{I}$. 

The generalized Hukuhara difference (gH-difference) of two intervals $T_1$ and $T_2$ was introduced by Stefanini and Bede \cite{stefanini}. This concept is represented as follows:
$$T_1 \ominus_{gH} T_2 = T_3 ~ \Leftrightarrow ~ \left \{ \begin{array}{rll}
	(i) &T_1 = T_2 + T_3, &\text{or} \\ 
	(ii) &T_2 = T_1 - T_3. &  
\end{array} \right.$$
In case $(i)$, the gH-difference coincides with the H-difference \cite{wu1}. For any two intervals $T_1=[t_1^l,t_1^u], ~ T_2 = [t_2^l, t_2^u], ~ T_1 \ominus_{gH} T_2$ always exists and is unique. Also, we have 
$$T_1\ominus_{gH} T_1 = [0,0] ~~~ \text{and} ~~~ T_1 \ominus_{gH} T_2 = [\text{min}\{t_1^l-t_2^l, t_1^u-t_2^u\}, \text{max}\{t_1^l-t_2^l, t_1^u-t_2^u\}].$$

The following lemma expresses the gH-difference of two intervals in $\mathbb{I}$ in terms of their center and half-width.

\begin{lemma} \cite{hilal}
	\label{lemma13,p2}
	For any two intervals $T_1,T_2 \in \mathbb{I}$ with $T_1=[t_1^l,t_1^u]=\langle t_1^c,t_1^w\rangle$ and $T_2=[t_2^l,t_2^u]=\langle t_2^c,t_2^w\rangle$, we have
	$$T_1\ominus_{gH} T_2~ = ~\langle t_1^c-t_2^c,~|t_1^w-t_2^w|\rangle$$
\end{lemma}

The order relation between two intervals in $\mathbb{I}$ used in the articles \cite{chalco-cano,chen,bilal1,bilal2,wu1,wu2} is a partial order relation given by
\begin{equation}\label{eq2,p2}
	T_1 \preceq_{lu} T_2 \Longleftrightarrow t_1^l\leq t_2^l ~ \text{and} ~ t_1^u\leq t_2^u.
\end{equation}
The order relation (\ref{eq2,p2}) in $\mathbb{I}$ is not a total order meaning that any two intervals in $\mathbb{I}$ are not comparable. For example, choose $T_1=[1,4]$ and $T_2=[2,3]$ then $t_1^l<t_2^l$ but $t_1^u> t_2^u$ which implies $A$ and $B$ are not comparable with respect to order relation \ref{eq2,p2}. Hence, it is not a total order relation.

In view of the above discussion, Bhunia and Samanta \cite{bhunia} proposed the following order relations:
\begin{itemize}
	\item [i)] {\large Minimization IVOP problem}\\
	For any two intervals $T_1,T_2 \in \mathbb{I}$ with $T_1=\langle t_1^c,t_1^w\rangle$ and $T_2=\langle t_2^c,t_2^w\rangle$, we say $T_1$ is superior (or more preferable) to $T_2$ in a minimization problem if and only if center of $T_1$ is strictly less than center of $T_2$ and half-width (radius), which measures uncertainty (or inexactness), of $T_1$ is less than or equal to $T_2$ i.e.,
	\begin{equation} \label{orderrelation,p2}
		T_1 \leq^{\text{min}} T_2 \Longleftrightarrow \begin{cases}
			t_1^c<t_2^c, & t_1^c \neq t_2^c;\\
			t_1^w\leq t_2^w, & t_1^c=t_2^c.
		\end{cases}
	\end{equation}
	$$~~T_1 <^{\text{min}} T_2 \Longleftrightarrow T_1 \leq^{\text{min}} T_2 \text{~and~} T_1\neq T_2.$$
	
	\item [ii)] {\large Maximization IVOP problem}\\
	Similar to minimization IVOP problem, the order relations in case of maximization IVOP are given by
	\begin{equation} \label{eq4,p2}
		T_1 \geq^{\text{max}} T_2 \Longleftrightarrow \begin{cases}
			t_1^c>t_2^c, & t_1^c \neq t_2^c;\\
			t_1^w\leq t_2^w, & t_1^c=t_2^c.
		\end{cases}
	\end{equation}
	$$~~T_1 >^{\text{max}} T_2 \Longleftrightarrow T_1 \geq^{\text{max}} T_2 \text{~and~} T_1\neq T_2.$$
\end{itemize} 
One can easily verify that the order relations given by the expressions (\ref{orderrelation,p2}) and (\ref{eq4,p2}) are total order relations. Throughout the paper, we will be considering the minimization IVOP problem and the order relation given by expression (\ref{orderrelation,p2}).\vspace{0.15cm}

In view of the order relation (\ref{orderrelation,p2}), we provide some basic lemmas that will be used frequently in sequel.

\begin{lemma}\cite{hilal} \label{lemma4,p2}
	For any two intervals $T_1,T_2\in \mathbb{I}$ with $T_1=\langle t_1^c,t_1^w\rangle$ and $T_2=\langle t_2^c,t_2^w\rangle$, and any $s_1, s_2 \in \mathbb{R}$, we have
	$$s_1 T_1 + s_2 T_2 =\langle s_1 t_1^c + s_2 t_2^c, |s_1|t_1^w +|s_2|t_2^w\rangle.$$
\end{lemma}

\begin{lemma} \label{lemma2,p2}
	For any $T_1,T_2,T_3,T_4 \in \mathbb{I}$, we have the following:
	\begin{itemize}
		\item [\bf (i)] if $T_1 \leq^{\text{min}} T_2$, then $s T_1 \leq^{\text{min}} s T_2$, ~ $s \geq 0$;
		\item [\bf (ii)] if $T_1 \leq^{\text{min}} T_2$ and $T_3 \leq^{\text{min}} T_4$, then $T_1+T_3 \leq^{\text{min}} T_2+T_4$;
		\item [\bf (iii)] if ~$0 \leq^{\text{min}} T_1+s$, then $-s \leq^{\text{min}} T_1$, ~ for any $s \in \mathbb{R}$.
	\end{itemize}
\end{lemma}
\begin{proof}
	{\bf(i)} From the order relation (\ref{orderrelation,p2}), we have 
	$$T_1 \leq^{\text{min}} T_2 \Longleftrightarrow \begin{cases}
		t_1^c<t_2^c, & t_1^c \neq t_2^c;\\
		t_1^w\leq t_2^w, & t_1^c=t_2^c,
	\end{cases}$$
	which for $s \geq 0$ gives
	$$\begin{cases}
		s t_1^c<s t_2^c, & s t_1^c \neq s t_2^c;\\
		s t_1^w\leq s t_2^w, & s t_1^c=s t_2^c.
	\end{cases}$$
	Using order relation (\ref{orderrelation,p2}), we have
	$$\langle s t_1^c,s t_1^w\rangle \leq^{\text{min}} \langle s t_2^c,s t_2^w\rangle.$$
	This together with Lemma \ref{lemma4,p2}, yields
	$$s T_1  \leq^{\text{min}} s T_2.$$
	{\bf (ii)} From order relation (\ref{orderrelation,p2}), we have 
	\begin{align*}
		&T_1 \leq^{\text{min}} T_2 \Longleftrightarrow \begin{cases}
			t_1^c<t_2^c, & t_1^c \neq t_2^c;\\
			t_1^w\leq t_2^w, & t_1^c=t_2^c,
		\end{cases}\\
		\text{and}~~&T_3 \leq^{\text{min}} T_4 \Longleftrightarrow \begin{cases}
			t_3^c<t_4^c, & t_3^c \neq t_4^c;\\
			t_3^w\leq t_4^w, & t_3^c=t_4^c,
		\end{cases}
	\end{align*}
	We now have the following four cases:\vspace{0.15cm}\\
	{ Case (1):} $t_1^c<t_2^c,~ t_1^c \neq t_2^c$ and $t_3^c<t_4^c, ~t_3^c \neq t_4^c.$
	$$\Rightarrow~ t_1^c+t_3^c<t_2^c+t_4^c,~~ t_1^c+t_3^c\neq t_2^c+t_4^c,\hspace{4.1cm}$$
	which, together with order relation (\ref{orderrelation,p2}) and Lemma \ref{lemma4,p2}, yields
	$$T_1+T_3\leq^{\text{min}}T_2+T_4.$$
	{ Case (2):} $t_1^c<t_2^c,~ t_1^c \neq t_2^c$ and $t_3^w\leq t_4^w, ~t_3^c = t_4^c.$
	$$\Rightarrow~ t_1^c+t_3^c<t_2^c+t_4^c,~~ t_1^c+t_3^c\neq t_2^c+t_4^c,\hspace{4.4cm}$$
	which, together with order relation (\ref{orderrelation,p2}) and Lemma \ref{lemma4,p2}, yields
	$$T_1+T_3\leq^{\text{min}}T_2+T_4.$$
	{Case (3):} $t_1^w\leq t_2^w,~ t_1^c = t_2^c$ and $t_3^c < t_4^c, ~t_3^c \neq t_4^c.$\\
	{Case (4):} $t_1^w\leq t_2^w,~ t_1^c = t_2^c$ and $t_3^w\leq t_4^w, ~t_3^c = t_4^c.$\\
	{Cases (3) and (4)} are similar to that of {Cases (1) and (2)}.
	\vspace{0.15cm}
	
	\noindent {\bf (iii)} Here $s =\langle s,~0\rangle$ is an interval with center $s$ and half-width (radius) equal to zero.
	$$0 \leq^{\text{min}} T_1+s$$
	$$\Rightarrow~~ \langle0,~0\rangle ~\leq^{\text{min}}~ \langle t_1^c,~t_1^w\rangle + \langle s,~0\rangle$$
	which, together with Lemma \ref{lemma4,p2}, yields
	$$\langle 0,~0\rangle ~\leq^{\text{min}} ~\langle t_1^c+s,~t_1^w\rangle.$$
	From order relation (\ref{orderrelation,p2}), we have two cases:\vspace{0.15cm}\\
	\noindent Case (a): $0 < t_1^c + s$ $\Rightarrow~ -s < t_1^c$
	\begin{align*}
		\Rightarrow~~\langle -s,~0\rangle ~&\leq^{\text{min}} ~\langle t_1^c,~t_1^w\rangle,\\
		i.e.,~~~~ -s~ &\leq^{\text{min}} T_1
	\end{align*}
	
	\noindent Case (b): $0=t_1^c+s,$ which implies that $t_1^c=-s.$ Also, $t_1^w \geq 0$. This gives
	$$\langle -s,~0\rangle ~\leq^{\text{min}}~ \langle -s,~t_1^w\rangle ~=~ \langle t_1^c,~t_1^w\rangle,$$
	$$i.e., ~~~~ -s \leq^{\text{min}} T_1.$$
	\end{proof}

A function $f:E \rightarrow \mathbb{I}$ defined on a subset $E\subseteq (M,g)$ is called an interval-valued function (IVF) and we write $f(p)=\langle f^c(p),f^w(p)\rangle$, where $f^c(p)$ (center function) and $f^w(p)$ (half-width or radius function) are real-valued functions defined on $E$, and satisfies $f^w(p) \geq0 ~ \forall p \in E.$ \vspace{0.2cm} 

Next, we consider the following IVOP problem on $(M,g)$,\\
$$\begin{array}{lll}
	(P_1)\hspace{1cm} & \text{minimize} & f(p)=~\langle f^c(p),f^w(p)\rangle\vspace{0.1cm}\hspace{2cm}\\ 
	& \text{subject to} & p \in X,
\end{array}$$
where $f:E\rightarrow \mathbb{I}$, $E\subseteq(M,g)$ and $X$ is the feasible set. \vspace{0.15cm}

In view of the order relation given by expression (\ref{orderrelation,p2}), we give some basic definitions which will be used in sequel. 

\begin{definition}
	\rm A feasible point $p_0 \in X$ is said to be an optimal solution (strict optimal solution) to IVOP problem $(P_1)$ if no $p\in X$ exists such that $f(p)<^{\text{min}} f(p_0)$ ($f(p)\leq^{\text{min}} f(p_0)$).
\end{definition}

\begin{definition} \rm
	A point $p_0 \in E$ is said to be a local minimum point (local strict minimum point) of an IVF $f:E\rightarrow \mathbb{I}$ with $f(p)=\langle f^c(p),f^w(p)\rangle ,$ defined on a nonempty subset $E\subseteq (M,g)$, if there exists $\delta > 0$ such that $f(p_0) \leq^{\text{min}}f(p)$ ($f(p_0) <^{\text{min}} f(p)$), $\forall ~ p \in B(p_0,\delta)\cap E$, where $B(p_0,\delta)$ is an open ball about $p_0$ of radius $\delta$.
\end{definition}

\begin{definition} \rm
	A point $p_0 \in E$ is said to be a global minimum point (global strict minimum point) of an IVF $f:E\rightarrow \mathbb{I}$ with $f(p)=\langle f^c(p),f^w(p)\rangle ,$ defined on a nonempty subset $E\subseteq (M,g)$, if $f(p_0) \leq^{\text{min}}f(p)$ ($f(p_0) <^{\text{min}} f(p)$), $\forall ~ p \in E$.
\end{definition}
One can similarly define local maximum point, local strict maximum point, global maximum point and global strict maximum point.

\section{Convexity of an IVF}
In this section, we provide some fundamental definitions and results related to an IVF which is convex on whole of its domain.

\begin{definition} \rm
	\cite{udriste} A subset $E \subseteq (M,g)$ is said to be {totally convex} if $E$ contains every geodesic $\gamma_{\1}$ of $M$ whose end points $p$ and $q$ are in $E$.
\end{definition} 

The following definition gives notion of convexity for a real-valued function defined on a totally convex set $E\subseteq(M,g)$.
\begin{definition} \rm \label{definitionconvexreal,p3}
	\cite{udriste} Suppose $f:E\rightarrow\mathbb{R}$ be a real-valued function defined on a totally convex set $E\subseteq(M,g)$. Then:
	\begin{itemize}
		\item [1)] $f$ is convex on $E$ if
		$$f(\gamma_{\1}(s)) \leq (1-s)f(p) +sf(q), ~~~ \forall~ p,q \in E,~~ \gamma_{\1}\in \Gamma,~~\forall~s\in[0,1],$$
		where $\Gamma$ is the collection of geodesics joining $p$ and $q$.
		\item [2)] $f$ is strictly convex on $E$ if
		$$f(\gamma_{\1}(s)) < (1-s)f(p) +sf(q), ~~~ \forall~ p,q \in E, ~~p\neq q,~~ \gamma_{\1}\in \Gamma,~~\forall~s\in(0,1).$$
		\item [3)] $f$ is linear affine on $E$ if
		$$f(\gamma_{\1}(s)) = (1-s)f(p) +sf(q), ~~~ \forall~ p,q \in E, ~~ \gamma_{\1}\in \Gamma,~~\forall~s\in[0,1].$$
	\end{itemize}
\end{definition}

The following definition extends the Definition \ref{definitionconvexreal,p3} to an IVF.
\begin{definition}\cite{hilal} \label{definitionconvexinterval,p3}
	\rm Suppose $f:E\rightarrow\mathbb{I}$ be an IVF defined on a totally convex set $E\subseteq M$. Then:
	\begin{itemize}
		\item [1)] $f$ is convex on $E$ if
		$$f(\gamma_{\1}(s)) \leq^{\text{min}} (1-s)f(p) +sf(q), ~~~ \forall~ p,q \in E,~~ \gamma_{\1}\in \Gamma,~~\forall~s\in[0,1],$$
		where $\Gamma$ is the collection of geodesics joining $p$ and $q$.
		\item [2)] $f$ is strictly convex on $E$ if
		$$f(\gamma_{\1}(s)) <^{\text{min}} (1-s)f(p) +sf(q), ~~~ \forall~ p,q \in E, ~~p\neq q,~~ \gamma_{\1}\in \Gamma,~~\forall~s\in(0,1).$$
	\end{itemize}
\end{definition}

\begin{definition}
	\rm An IVF $f:E\rightarrow \mathbb{I}$, defined on a totally convex set $E\subseteq(M,g)$, is said to be linear affine on $E$ if
	$$f(\gamma_{\1}(s)) = (1-s)f(p) + sf(q), ~~~ \forall~ p,q \in E, ~~ \gamma_{\1}\in \Gamma,~~\forall~s\in[0,1].$$
\end{definition}

\begin{example} \label{example3.1,p2}
	\rm The set $S^n_{++}$ of $n \times n$ symmetric positive definite matrices with entries from $\mathbb{R}$ is a Hadamard manifold with Riemannian metric:
	$$g_p(X,Y) = Tr(p^{-1}Xp^{-1}Y), ~~~ \forall~ p \in S^n_{++}, ~~~ X,Y \in T_p(S^n_{++}).$$
	The unique minimal geodesic joining $p, q \in S^n_{++}$ is given by 
	$$\gamma(s) = p^{\frac{1}{2}}(p^{-\frac{1}{2}}qp^{-\frac{1}{2}})^sp^{\frac{1}{2}}, ~~~ \forall ~ s \in [0,1].$$
	For more details, one can refer to \cite{bacak,serge}.
	
	Define $f: S^n_{++} \rightarrow \mathbb{I}$, as follows
	$$f(p)= \langle \ln(\det(p)),~(\ln(\det(p)))^2\rangle.$$
	Now, for any $p,q \in S^n_{++}$, 
	\begin{align} \label{eq6a,p2} 
		f^c(\gamma(s)) &= \ln(\det(\gamma(s))) \notag\\
		&= \ln(\det(p^{\frac{1}{2}}(p^{-\frac{1}{2}}qp^{-\frac{1}{2}})^sp^{\frac{1}{2}})) \notag\\
		&= \ln(\det(p)) + s(\ln(\det(q)) - \ln(\det(p)))\notag\\
		&= (1-s)\ln(\det(p)) + s\ln(\det(q))\notag\\
		&= (1-s)f^c(p) + sf^c(q).
	\end{align}
	This shows that $f^c$ is linear affine on $S^n_{++}$. One can similarly show that
	\begin{equation} \label{eq6b,p2}
		f^w(\gamma(s)) < (1-s)f^w(p) + sf^w(q), ~~\forall~p,q \in S^n_{++},
	\end{equation}
	i.e., $f^w$ is strictly convex on $S^n_{++}$.\\
	From (\ref{eq6a,p2}) and (\ref{eq6b,p2}), together with order relation (\ref{orderrelation,p2}) and Lemma \ref{lemma4,p2}, we have
	$$f(\gamma(s)) <^{\text{min}} (1-s)f(p) + sf(q), ~~\forall~p,q \in S^n_{++}.$$
	So, $f$ is strictly convex and hence convex on $S^n_{++}$.
	However, $f$ fails to be convex on $S^n_{++}$ in the usual sense. For this, let $n=2,~ p=I_2,~ q=2I_2$, where $I_2$ is $2\times 2$ identity matrix, and $s=\frac{1}{2}$, then
	\begin{align*}
		f(\frac{1}{2}I_2 + \frac{1}{2}(2I_2)) &= ~\langle 0.811,~0.658\rangle\\
		\text{and} ~~ \frac{1}{2}f(I_2) + \frac{1}{2}f(2I_2) &=~ \langle 0.693,~0.48\rangle.
	\end{align*}
	Clearly, 
	$$f(\frac{1}{2}I_2 + \frac{1}{2}(2I_2)) >^{\text{min}} \frac{1}{2}f(I_2) + \frac{1}{2}f(2I_2).$$
	This shows $f$ fails to be convex on $S^n_{++}$ in the usual sense.
\end{example}

The following two lemmas in the sequel give sufficient conditions for an IVF to be convex on its totally convex domain $E\subseteq (M,g)$.

\begin{lemma} \label{lemma1*,p2}
	Suppose that an IVF $f:E \rightarrow \mathbb{I}$ with $f(p)=\langle f^c(p),f^w(p)\rangle$ be defined on a totally convex set $E\subseteq (M,g)$. If the center function $f^c$ is strictly convex on $E$, then the IVF $f$ is convex on $E$.
\end{lemma}
\begin{proof}
	Since the center function $f^c$ is strictly convex on $E$, we have for any $p,q \in E$ that
	$$f^c(\gamma_{\1}(s)) < (1-s)f^c(p) +sf^c(q), ~~~ p\neq q,~~ \gamma_{\1}\in \Gamma,~~\forall~s\in(0,1).$$
	This from order relation (\ref{orderrelation,p2}), yields that
	$$\langle f^c(\gamma_{\1}(s)),~f^w(\gamma_{\1}(s))\rangle \leq^{\text{min}} \langle  (1-s)f^c(p) +sf^c(q),~(1-s)f^w(p) +sf^w(q)\rangle,$$
	which by Lemma \ref{lemma4,p2}, gives
	$$f(\gamma_{\1}(s)) \leq^{\text{min}} (1-s)f(p)+sf(q),~~\forall~p,q \in E,~~\gamma_{\1}\in \Gamma,~~\forall~s\in[0,1].$$
	\end{proof}

\begin{lemma}
	Suppose that an IVF $f:E \rightarrow \mathbb{I}$ with $f(p)=\langle f^c(p),f^w(p)\rangle$ be defined on a totally convex set $E\subseteq (M,g)$. If the center function $f^c$ is linear affine on $E$ and the half-width function $f^w$ is convex on $E$, then the IVF $f$ is convex on $E$.
\end{lemma}
\begin{proof}
	The proof is similar to Lemma \ref{lemma1*,p2}.
	\end{proof}

The following lemma gives necessary condition for an IVF function to be convex on $E \subseteq (M,g)$.

\begin{lemma}
	If an IVF $f:E \rightarrow \mathbb{I}$ with $f(p)=\langle f^c(p),f^w(p)\rangle$, defined on a totally convex set $E\subseteq (M,g)$, is convex on $E$, then the center function $f^c$ is convex on $E$.
\end{lemma}
\begin{proof}
	The proof follows directly from order relation (\ref{orderrelation,p2}) and Lemma \ref{lemma4,p2}.
	\end{proof}

We remark here that convexity of an IVF $f$ on $E$ doesn't necessarily imply the convexity of half-width function $f^w$. For counter example, one can refer to Example \ref{example6*,p2}.

The next definition gives the notion of convexity of an IVF in terms of its central and half-width functions. We call such convexity as cw-convexity.
\begin{definition} \rm 
	An IVF $f:E\rightarrow\mathbb{I}$ with $f(p)=\langle f^c(p),~f^w(p)\rangle$, defined on a totally convex set $E\subseteq (M,g)$, is cw-convex (strictly cw-convex) on $E$ if $f^c$ and $f^w$ are convex (strictly convex) on $E$.
\end{definition}

The following lemma shows that cw-convexity implies the convexity of an IVF.
\begin{lemma} \label{lemma3.4,p2}
	Suppose that an IVF $f:E \rightarrow \mathbb{I}$ with $f(p)=\langle f^c(p),f^w(p)\rangle$, defined on a totally convex set $E\subseteq (M,g)$, is cw-convex on $E$, then $f$ is convex on $E$.
\end{lemma}
\begin{proof}
	Since $f$ is cw-convex on $E$, both $f^c$ and $f^w$ are convex on $E$. For any $p,q \in E$ and any geodesic $\gamma_{\1}(s), s\in [0,1],$ with $\gamma_{\1}(0)=p$ and $\gamma_{\1}(1)=q$, we have
	\begin{equation} \label{eq6,p2}
		\begin{split}
			f^c(\gamma_{\1}(s)) &\leq (1-s)f^c(p) + sf^c(q)\\
			\text{and~~~} f^w(\gamma_{\1}(s)) &\leq (1-s)f^w(p) + sf^w(q)
		\end{split}
	\end{equation}
	Let $T=\{s \in [0,1]: f^c(\gamma_{\1}(s)) = (1-s)f^c(p) + sf^c(q)\}$. Then from (\ref{eq6,p2}), we can deduce that
	\begin{align*}
		f^c(\gamma_{\1}(s)) &< (1-s)f^c(p) + sf^c(q), ~~~\forall ~ s\in[0,1]\setminus T\\
		\text{and~~~} f^w(\gamma_{\1}(s)) &\leq (1-s)f^w(p) + sf^w(q),~~\forall ~ s\in T,
	\end{align*}
	which together with Lemma \ref{lemma4,p2} and order relation (\ref{orderrelation,p2}), yields 
	$$f(\gamma_{\1}(s)) \leq^{\text{min}} (1-s)f(p) + sf(q),$$
	Hence, $f$ is convex.
	\end{proof}

The following example illustrates that the reverse implication of Lemma \ref{lemma3.4,p2} is not possible in general.

\begin{example} \label{example6*,p2} \rm
	Let $\mathit{M}=\{e^{i\theta} : \theta \in \mathbb{R}\}$ be a non-compact $1-$dimensional Riemannian manifold$^*$.\let\thefootnote\relax\footnotetext{\footnotesize
		$^*$ In this case, we assume that the manifold \( \mathit{M} = \{e^{i\theta} : \theta \in \mathbb{R}\} \) is not periodic, meaning that distinct values of \(\theta\) correspond to distinct points in \( \mathit{M} \). As a result, \( \mathit{M} \) is diffeomorphic to \(\mathbb{R}\) rather than the unit circle \( S^1 =\{(p,q)\in \mathbb{R}^2: p^2+q^2=1\} \). Unlike the standard compact circle where \(\theta\) is identified modulo \(2\pi\), our construction treats \(\theta\) as a global coordinate extending infinitely in both directions.
	} 
	The geodesic segment $\gamma_{\1}(s)$ joining $p=e^{i\theta}$ and $q=e^{i\phi}$ is given by 
	$$\gamma_{\1}(s)=e^{i((1-s)\theta+s\phi)}.$$
	Define $f:M \rightarrow \mathbb{I}$, as
	$$f(p) = \langle \theta^2,~-\theta^2+5\pi^2\rangle, ~~ p=e^{i\theta} \in M.$$
	The center function $f^c(p)=\theta^2$ is strictly convex on $M$ which is evident from the following:\\
	For any $p=e^{i\theta}, q=e^{i\phi}\in M$,~ $\gamma_{\1}(s), ~s\in[0,1],$ 
	\begin{align*}
		f^c(\gamma_{\1}(s)) &= f^c(e^{i((1-s)\theta+s\phi)})\\
		&= \left((1-s)\theta+s\phi\right)^2\\
		&< (1-s)\theta^2+s\phi^2~~ (\because ~f(p)=p^2~ \text{is strictly convex on}~\mathbb{R})\\
		&= (1-s)f^c(p)+sf^c(q).
	\end{align*}
	Also, one can similarly show that $f^w(p)=-\theta^2 +5\pi^2$ is not convex on $E$. So, $f$ is not cw-convex.\\
	However, from Lemma \ref{lemma1*,p2}, it follows that $f$ is convex.
\end{example}

Next, we present some of the basic results related to convexity of an IVF on Riemannian Manifolds.
\begin{lemma}
	An IVF $f:E\rightarrow \mathbb{I}$ with $f(p) = \langle f^c(p),~f^w(p)\rangle$, defined on a totally convex set $E\subseteq (M,g)$, is convex on $E$ if and only if $\forall ~p,q \in E$ the function $f \circ \gamma_{\1}$ is convex on [0,1], where $\gamma_{\1}$ is the geodesic segment joining $p$ and $q$.
\end{lemma}
\begin{proof}
	If $f \circ \gamma_{\1}$ is convex on [0,1], then, for any $s_1,s_2,s_3 \in[0,1]$, we have
	$$(f \circ \gamma_{\1})((1-s_3)s_1+s_3 s_2) \leq^{\text{min}} (1-s_3)(f \circ \gamma_{\1})(s_1)+s_3 (f \circ \gamma_{\1})(s_2).$$
	In particular for $s_1=0, ~ s_2=1$, we have
	\begin{align*}
		&(f \circ \gamma_{\1})(s_3) &&\hspace{-0.5cm}~ \leq^{\text{min}} (1-s_3)(f \circ \gamma_{\1})(0)+s_3 (f \circ \gamma_{\1})(1),\\
		&i.e.,~&& \\ 
		&f (\gamma_{\1}(s_3))&&\hspace{-0.5cm}~ \leq^{\text{min}} (1-s_3)f(p)+s_3 f(q), ~~\forall~p,q\in E,~\gamma_{\1}(s_3) \in \Gamma \text{and} ~s_3 \in [0,1],
	\end{align*}
	where $\Gamma$ is the collection of all geodesics joining $p$ and $q$.
	Hence, $f$ is convex on $E$.
	
	Conversely, suppose that $f$ is a convex function and $\gamma_{\1}(s),~ s\in[0,1],$  is the geodesic joining $p$ and $q$, then the restriction of $\gamma_{\1}$ to $[s_1,s_2]\subseteq[0,1]$ joins the points $\gamma_{\1}(s_1)$ and $\gamma_{\1}(s_2)$. We parameterize this restriction as,
	$$\alpha(s_3) = \gamma_{\1}(s_1 +s_3(s_2-s_1)), ~~~s_3 \in [0,1].$$
	From convexity of $f$, we have
	$$(f \circ \alpha)(s_3) \leq^{\text{min}} (1-s_3)(f \circ \alpha)(0) +s_3 (f \circ \alpha)(1)$$
	$$\Rightarrow~~ (f\circ\gamma_{\1})((1-s_3)s_1 +s_3 s_2) \leq^{\text{min}} (1-s_3)(f\circ\gamma_{\1})(s_1) + s_3 (f\circ\gamma_{\1})(s_2).$$
	Since $s_1, s_2 \in [0,1]$ are arbitrary, we conclude that $f\circ \gamma_{\1}$ is convex on [0,1].
	\end{proof}

\begin{lemma}
	Suppose that an IVF $f:E\rightarrow \mathbb{I}$ with $f(p) = \langle f^c(p),~f^w(p)\rangle$, defined on a totally convex set $E\subseteq (M,g)$, is convex on $E$, then the lower level set
	$$D:=\{p\in E: f(p) \leq^{\text{min}}B\},$$
	where $B$ is an interval in $\mathbb{I}$, is totally convex subset of $E$.
\end{lemma}
\begin{proof}
	Let $p,q \in D$ be arbitrary, then $f(p) \leq^{\text{min}}B$ and $f(q) \leq^{\text{min}}B$. Let $\gamma_{\1}(s), s\in[0,1]$ be the geodesic joining $p$ and $q$. Since $f$ is convex, we have
	$$f(\gamma_{\1}(s)) \leq^{\text{min}}(1-s)f(p)+sf(q)$$
	Using parts {\bf (i) and (ii)} of Lemma \ref{lemma2,p2}, the above expression yields
	$$f(\gamma_{\1}(s)) \leq^{\text{min}}(1-s)B+sB =^{\text{min}} B.$$
	This shows $D$ is totally convex subset of $E$.
	\end{proof}

\begin{lemma}
	Suppose that IVFs $f,g:E\rightarrow \mathbb{I}$ with $f(p) = \langle f^c(p),~f^w(p)\rangle$ and $g(p) = \langle g^c(p),~g^w(p)\rangle$, defined on a totally convex set $E\subseteq (M,g)$, are convex on $E$, then $\alpha f + \beta g$ is also convex on $E$, for any $\alpha,~\beta \geq0$.
\end{lemma}
\begin{proof}
	The proof follows directly from parts {\bf (i) and (ii)} of Lemma \ref{lemma2,p2}.
	\end{proof}
In the following lemma, we discuss an important characterization of interval-valued convex function in terms of its epigraph.
\begin{lemma}
	An IVF $f:E\rightarrow \mathbb{I}$ with $f(p) = \langle f^c(p),~f^w(p)\rangle$, defined on a totally convex set $E\subseteq (M,g)$, is convex on $E$ if and only if $\forall ~p,q \in E$ its epigraph 
	$$Epi(f) := \{(p,B) \in E\times \mathbb{I} : f(p) \leq^{\text{min}} B\},$$
	is a convex set.
\end{lemma}

\begin{proof}
	Suppose $f$ is convex on $E$ and $(p,B),(q,C) \in Epi(f)$, then
	$$f(p)\leq^{\text{min}} B ~~~\text{and}~~~ f(q)\leq^{\text{min}} C.$$
	By convexity of $f$ on $E$,
	$$f(\gamma_{\1}(s)) \leq^{\text{min}}(1-s)f(p)+sf(q),$$
	where $\gamma_{\1}(s), s\in[0,1]$ is a geodesic with $\gamma_{\1}(0)=p$ and $\gamma_{\1}(1)=q$. Using Lemma \ref{lemma2,p2}, we get
	$$f(\gamma_{\1}(s)) \leq^{\text{min}}(1-s)B+sC,$$
	$$\Rightarrow ~~(\gamma_{\1}(s), ~(1-s)B+sC) \in Epi(f).$$
	So, $Epi(f)$ is a totally convex set.
	
	Conversely, assume that $Epi(f)$ is totally convex. Let $p,q\in E$, then $(p,f(p)),~(q,f(q))\in Epi(f)$. By hypothesis, we have
	$$(\gamma_{\1}(s),~ (1-s)f(p)+sf(q))\in Epi(f),$$
	where, $\gamma_{\1}(s), s\in[0,1]$ is a geodesic  with $\gamma_{\1}(0)=p$ and $\gamma_{\1}(1)=q$,
	$$\Rightarrow~ f(\gamma_{\1}(s)) \leq^{\text{min}}(1-s)f(p)+sf(q).$$
	This yields that $f$ is convex on $E$.
	\end{proof}

\section{Convexity of an IVF at a point}

In this section, we provide some fundamental definitions and results related to an IVF which is convex at a point.

\begin{definition}
	\rm \cite{udriste} A set $E\subseteq (M,g)$ is said to be star-shaped at $p_0 \in E$ if $\gamma_{\2}(s)\in E$ whenever $p\in E$ and $s \in (0,1)$, where $\gamma_{\2}$ is any geodesic in $E$ joining $p_0$ with $p$.
\end{definition}

\begin{definition} \rm 
	\cite{udriste} Let $E \subseteq (M,g)$ be star-shaped at $p_0 \in E$ and $f:E\rightarrow\mathbb{R}$ be a real-valued function. Then:
	\begin{itemize}
		\item [1)] $f$ is convex at $p_0$ if
		$$f(\gamma_{\2}(s)) \leq (1-s)f(p_0) +sf(p), ~~~ \forall~ p \in E,~~ \gamma_{\2}\in \Gamma_0,~~\forall~s\in(0,1),$$
		where $\Gamma_0$ is the collection of all geodesics emanating from $p_0$ and terminating at $p$.
		\item [2)] $f$ is strictly convex at $p_0$ if
		$$f(\gamma_{\2}(s)) < (1-s)f(p_0) +sf(p), ~~~ \forall~ p \in E, ~~p\neq p_0,~~ \gamma_{\2}\in \Gamma_0,~~\forall~s\in(0,1).$$
		\item [3)] $f$ is linear affine at $p_0$ if
		$$f(\gamma_{\2}(s)) = (1-s)f(p_0) +sf(p), ~~~ \forall~ p \in E, ~~ \gamma_{\2}\in \Gamma_0,~~\forall~s\in(0,1).$$
	\end{itemize}
\end{definition}

Next, we present the definition of convexity at a point of an IVF defined on a set which is star-shaped at that point. 
\begin{definition} \cite{hilal}
	\rm Let $E \subseteq (M,g)$ be star-shaped at $p_0 \in E$ and $f:E\rightarrow\mathbb{I}$ be an IVF. Then:
	\begin{itemize}
		\item [1)] $f$ is convex at $p_0$ if
		$$f(\gamma_{\2}(s)) \leq^{\text{min}} (1-s)f(p_0) +sf(p), ~~~ \forall~ p \in E,~~ \gamma_{\2}\in \Gamma_0,~~\forall~s\in(0,1),$$
		where $\Gamma_0$ is the collection of all geodesics emanating from $p_0$ and terminating at $p$.
		\item [2)] $f$ is strictly convex at $p_0$ if
		$$f(\gamma_{\2}(s)) <^{\text{min}} (1-s)f(p_0) +sf(p), ~~~ \forall~ p \in E, ~~p\neq p_0,~~ \gamma_{\2}\in \Gamma_0,~~\forall~s\in(0,1).$$
	\end{itemize}
\end{definition}

\begin{definition}
	\rm Let $E\subseteq (M,g)$ be star-shaped at $p_0\in E$, then an IVF $f$ is linear affine at $p_0$ if
	$$f(\gamma_{\2}(s)) = (1-s)f(p_0) +sf(p), ~~~ \forall~ p \in E, ~~ \gamma_{\2}\in \Gamma_0,~~\forall~s\in(0,1).$$
\end{definition}

\begin{definition} \label{definition**,p2}
	\rm	Let $E\subseteq (M,g)$ be star-shaped at $p_0 \in E$. We say an IVF $f:E\rightarrow \mathbb{I}$ with $f(p)=\langle f^c(p),~f^w(p)\rangle$ is cw-convex at $p_0$ if the center function $f^c$ and half-width $f^w$ are convex at $p_0$.
\end{definition}

The following lemmas are similar to the ones presented in the previous section. Here we only provide the statements as the proofs are respectively similar to their analogous ones.

\begin{lemma}
	Suppose $E\subseteq (M,g)$ is star-shaped at $p_0 \in E$ and let $f:E \rightarrow \mathbb{I}$ be an IVF with $f(p)=\langle f^c(p),f^w(p)\rangle$. If the center function $f^c$ is strictly convex at $p_0\in E$, then the IVF $f$ is convex at $p_0\in E$.
\end{lemma}

\begin{lemma}
	Suppose $E\subseteq (M,g)$ is star-shaped at $p_0 \in E$ and let $f:E \rightarrow \mathbb{I}$ be an IVF with $f(p)=\langle f^c(p),f^w(p)\rangle$. If the center function $f^c$ is linear affine at $p_0$ and the half-width function $f^w$ are convex at $p_0$, then the IVF $f$ is convex at $p_0$.
\end{lemma}

\begin{lemma}
	Suppose $E\subseteq (M,g)$ is star-shaped at $p_0 \in E$. If the IVF $f:E \rightarrow \mathbb{I}$, with $f(p)=\langle f^c(p),f^w(p)\rangle$, is convex at $p_0$, then $f^c$ is also convex at $p_0$.
\end{lemma}

\begin{lemma}
	Suppose $E\subseteq (M,g)$ is star-shaped at $p_0 \in E$. If the IVF $f:E \rightarrow \mathbb{I}$, with $f(p) = \langle f^c(p),~f^w(p)\rangle$, is cw-convex at $p_0$, then $f$ is convex at $p_0$.
\end{lemma}

\begin{lemma} \label{lemma4**,p2}
	Suppose $E\subseteq (M,g)$ is star-shaped at $p_0 \in E$ and let $f:E \rightarrow \mathbb{I}$ be an IVF with $f(p) = \langle f^c(p),~f^w(p)\rangle$. Then, $f$ is convex at $p_0$ if and only if for any $p \in E$, the function $f \circ \gamma_{\2}:[0,1]\rightarrow\mathbb{I}$ is convex at 0, where $\gamma_{\2}$ is the geodesic segment joining $p_0$ and $p$.
\end{lemma}

The following example shows that the lower level sets of a real-valued function, which is convex at a single point, may not be a convex set in general.

\begin{example}\rm 
	The function $f(p)=-|p|,~p\in \mathbb{R}$ is convex at $p=0$ only. The lower level set, $D_a = \{p\in \mathbb{R}:f(p)\leq a\}, a\in \mathbb{R}$, at $a=-1$ is $D_{-1}=(-\infty,-1)\cup(1, \infty)$ which is not a convex set.
\end{example}

In view of above example, we have the following lemma.

\begin{lemma} \label{lemma4.6,p3}
	Let $E\subseteq (M,g)$ be star-shaped at $p_0 \in E$ and $f:E \rightarrow \mathbb{R}$ be convex at $p_0$, then the lower level set
	$$D_a = \{p\in \mathbb{E}:f(p)\leq a\}, ~~a \in \mathbb{R},$$
	is star-shaped at $p_0$ if $p_0 \in D_{a}$.
\end{lemma}
\begin{proof}
	For any $p\in D_a$ with $p_0\in D_{a}$, we have 
	$$f(p)\leq a ~~~ \text{and} ~~~ f(p_0)\leq a.$$
	Let $\gamma_{\2}(s), s\in[0,1]$ be any geodesic joining $p_0$ with $p$. From convexity of $f$ at $p_0$, we have
	\begin{align*}
		f(\gamma_{\2}(s)) &\leq (1-s)f(p_0)+sf(p),\\
		&\leq (1-s)a +s(a),\\
		&=a.
	\end{align*}
	which shows $D_a$ is star-shaped at $p_0$.
	\end{proof}

For the case of convex IVF at a point, the above Lemma is stated as follows:

\begin{lemma}\label{lemma4.7,p3}
	Let $E\subseteq (M,g)$ be star-shaped at $p_0 \in E$ and the IVF $f:E \rightarrow \mathbb{I}$ be convex at $p_0$, then the lower level set
	$$D= \{p\in \mathbb{E}:f(p)\leq^{\text{min}} B\}, ~~B \in \mathbb{I},$$
	is star-shaped at $p_0$ if $p_0 \in D.$
\end{lemma}
\begin{proof}
	The proof is analogous to the Lemma \ref{lemma4.6,p3}.
	\end{proof}

\section{Optimality conditions for an unconstrained optimization programming problem}

In this section, we present some optimality conditions for an unconstrained IVOP problem.

\begin{definition}\cite{hilal} \label{definition5.1,p2} \rm 
	Let $E$ be a subset of $(M,g)$ and $p\in E$. Let $X_p \in T_p(E)$ and $\gamma(s);~s\in I, ~0\in I~ \&~ \gamma(I)\subseteq E$, be a geodesic for which $\gamma(0)=p, ~ \&~ \dot{\gamma}(0)=X_p$. We say a real-valued function $f:E \rightarrow \mathbb{R}$ is directionally differentiable at $p$ in the direction $X_p$, if the limit
	$$Df(p;~X_p)=\lim\limits_{s\rightarrow0^+} \frac{f(\gamma(s))-f(p)}{s}$$
	exists, where $Df(p;X_p)$ is said to be directional derivative of $f$ at $p$ in the direction $X_p$. Moreover, we say $f$ is directionally differentiable at $p$, if $Df(p;X_p)$ exists at $x$ in every direction $X_p \in T_p(E)$. Furthermore, if $Df(p;X_p)$ exists at each $p\in E$ and in every direction $X_p \in T_p(E)$, we say $f$ is directionally differentiable on $E$.
\end{definition}

\begin{theorem} \cite{hilal} \label{theorem30,p2}
	Let $E \subseteq (M,g)$ be star-shaped at $p_0 \in E$ and the function $f:E\rightarrow \mathbb{R}$ be directionally differentiable at $p_0$,
	\begin{enumerate} 
		\item [(i)] if $f$ is convex at $p_0$, then 
		\begin{equation}
			f(p)-f(p_0)\geq Df(p_0;~X_{p_0}); ~~\forall~p \in E,~~ \forall~\gamma_{\2}\in \Gamma_0,
		\end{equation} \label{eq9,p2}
		where $\Gamma_0$ is the set of geodesics joining $p_0$ and $p$ such that $\gamma_{\2}(0)=p$ and $\dot{\gamma}_{\2}(0)=X_{p_0}$,
		\item [(ii)] if $f$ is strictly convex on $E$ then 
		$$f(p)-f(p_0)> Df(p_0;~X_{p_0}); ~~\forall~p \in E,~~ p\neq p_0~~ \forall~\gamma_{\2}\in \Gamma_0.$$
	\end{enumerate}
\end{theorem}

As an immediate outcome of Theorem \ref{theorem30,p2}, the following corollary provides both a necessary and sufficient condition for a point to be a local minimum.

\begin{corollary}
	Suppose $E \subseteq (M,g)$ is star-shaped at $p_0 \in E$ and the function $f:E\rightarrow \mathbb{R}$ be directionally differentiable and convex at $p_0$. Then $p_0$ is a local minimum point of the real-valued function $f$ if and only if $Df(p_0;X_{p_0}) \geq 0,~ \forall~X_{p_0} \in T_{p_0}(E)$
\end{corollary}
\begin{proof}
	The proof follows directly from Theorem \ref{theorem30,p2}.
	\end{proof}

\begin{theorem}
	Suppose that $E\subseteq (M,g)$ is star-shaped at $p_0\in E$ and an IVF $f:E \rightarrow \mathbb{I}$ with $f(p)=\langle f^c(p), f^w(p) \rangle$ is convex at $p_0$. If $p_0$ is a local minimum point of $f$, then $p_0$ is also a global minimum point of $f$.
\end{theorem}

\begin{proof}
	Since $p_0$ is a local minimum point, $\exists~\epsilon>0$ such that, $f(p_0)\leq f(p); ~ \forall~ p \in B(p_0;\epsilon)\cap E$, where $B(p_0;\epsilon)$ is an open $\epsilon$-ball about $p_0$. Suppose  $\exists$ $q\in E$ such that $f(q)<f(p_0)$. We consider the geodesic $\gamma_{\6}:[0,1]\rightarrow E$ with $\gamma_{\6}(0)=p_0$ and $\gamma_{\6}(1)=q$. From convexity of $f$ at $p_0$, we have 
	$$f(\gamma_{\6}(s))\leq^{\text{min}} (1-s)f(p_0)+sf(q),$$
	which from parts {\bf (i)} and {\bf(ii)} of Lemma \ref{lemma2,p2} and transitivity of order relation (\ref{orderrelation,p2}), yields
	$$f(\gamma_{\6}(s))<^{\text{min}}f(p_0).$$
	But, $\gamma_{\6}(s) \in B(p_0;\epsilon)\cap E$, for some $s\in (0,1)$ which gives a contradiction and hence, we conclude
	$$f(p_0)\leq^{\text{min}} f(q), ~ \forall~ q\in E.$$
	\end{proof}

\begin{remark}
	\rm The minimum value of an IVF, convex at that point where the minimum is attained, remains the same, provided it exists.
\end{remark}

\begin{theorem} \label{theorem5.3,p3}
	Suppose that $E\subseteq (M,g)$ is star-shaped at $p_0\in E$ and an IVF $f:E \rightarrow \mathbb{I}$ is convex at $p_0$. If $p_0$ is local minimum point of $f$, then the collection $K$ of minimum points of $f$ is star-shaped at $p_0$.
\end{theorem}

\begin{proof}
	Suppose that $B\in \mathbb{I}$ is the minimum value of $f$, then we can express $K$ as $K=E\cap D$, where $D=\{p\in E: f(p)\leq^{\text{min}}B\}$ is star-shaped at $p_0$ by Lemma \ref{lemma4.7,p3}. One can easily show that intersection of two star-shaped sets, which are star-shaped at a common point say $p_0$, is also star-shaped at $p_0$. Hence, we conclude that K is also star-shaped at $p_0$.
	\end{proof}

\begin{corollary}
	Suppose that $E\subseteq (M,g)$ is star-shaped at $p_0\in E$ and an IVF $f:E \rightarrow \mathbb{I}$ is convex at $p_0$. Let $p_0$ be a local minimum point of $f$ and $K$ be the collection of minimum points of $f$. If $K$ contains any point other than $p_0$, then $K$ is an infinite set and $f$ can not be strictly convex at $p_0$.
\end{corollary}

\begin{proof}
	By Theorem \ref{theorem5.3,p3}, if  $p \in K$ be a point other than $p_0$ then every point on the geodesic $\gamma_{\2}(s),~s\in[0,1]$, joining $p_0$ and $p$ is also a minimum point of $f$ and hence, $K$ can't be finite. Also, $f(\gamma_{\2}(s))=f(p)=f(p_0), ~\forall~s\in[0,1].$ So, $f$ can not be strictly convex.
	\end{proof}

\begin{definition} \cite{hilal} \rm 
	Let $E$ be a subset of $(M,g)$ and $p\in E$. Let $X_p \in T_p(E)$ and $\gamma(s);~s\in I, ~0\in I~ \&~ \gamma(I)\subseteq E$, be a geodesic for which $\gamma(0)=p, ~ \&~ \dot{\gamma}(0)=X_p$. We say an IVF $f:E \rightarrow \mathbb{I}$ with $f(p)= \langle f^c(p),~f^w(p)\rangle$ is gH-directionally differentiable at $p$ in the direction $X_p$, if the limit
	$$Df(p;~X_p)=\lim\limits_{s\rightarrow0^+} \frac{f(\gamma(s))\ominus_{gH}f(p)}{s}$$
	exists, where $Df(p;X_p)$ is said to be gH-directional derivative of $f$ at $p$ in the direction $X_p$. Moreover, we say $f$ is gH-directionally differentiable at $p$, if $Df(p;X_p)$ exists at $p$ in every direction $X_p \in T_p(E)$. Furthermore, if $Df(p;X_p)$ exists at each $p\in E$ and in every direction $X_p \in T_p(E)$, we say $f$ is gH-directionally differentiable on $E$.
\end{definition}

The following lemma gives the equivalence of gH-directional differentiability of an IVF in terms of its center function and half-width function.

\begin{lemma} \cite{hilal} \label{lemma22,p2}
	Let $f:E\rightarrow \mathbb{I}$ with $f(p)= \langle f^c(p),~f^w(p)\rangle$ be an IVF defined on $E\subseteq (M,g)$. Let $p\in E$ and $\gamma(s);~s\in I, ~0\in I~ \text{and}~ \gamma(I)\subseteq E$, be any geodesic for which $\gamma(0)=p, ~ \&~ \dot{\gamma}(0)=X_p\in T_{p}(E)$ such that $(f^w \circ \gamma)(s)$ is non-decreasing for $s \in I\cap[0,\infty)$. Then, gH-directional derivative of $f$ exists at $p$ in the direction $X_p$ if and only if the directional derivative of $f^c$ and $f^w$ exists at $p$ in the direction $X_p$. Hence,
	$$Df(p;X_p)~ = ~\langle Df^c(p;X_p),~Df^w(p;X_p)\rangle.$$
	where $Df^w(p;X_p)\geq 0$.
\end{lemma}

\begin{theorem}\cite{hilal} \label{theorem23,p2}
	Let $E\subseteq (M,g)$ be star-shaped at $p_0 \in E$ and $f:E\rightarrow \mathbb{I}$ be an IVF with $f(p)=\langle f^c(p), f^w(p)\rangle$. Let $\gamma(s);~s\in I, ~0\in I~ \&~ \gamma(I)\subseteq E$, be  geodesic for which $\gamma(0)=p_0, ~ \&~ \dot{\gamma}(0)=X_{p_0}\in T_{p_0}(E)$ such that $(f^w \circ \gamma)(s)$ is non-decreasing for $s\in [0, \infty)$. Suppose that $f$ is gH-directionally differentiable at $p_0$,
	\begin{enumerate}
		\item [i)] if $f$ is cw-convex  at $p_0$, then
		\begin{equation}\label{eq26,p2}
			Df(p_0;X_{p_0}) \leq^{\text{min}} f(p)\ominus_{gH}f(p_0), ~~ \forall ~ p\in E ~~\text{and}~~ \forall ~\gamma_{\2}\in \Gamma_0,
		\end{equation}
		where $\Gamma_0$ is collection of geodesics joining $p_0$ and $p$ such that $\gamma_{\2}(0)=p_0$ and $\dot{\gamma}_{\2}(0)=X_{p_0}\in T_{p_0}(E).$
		\item [ii)] if $f$ is strictly cw-convex  at $p_0$, then
		$$Df(p_0;X_{p_0}) <^{\text{min}} f(p)\ominus_{gH}f(p_0), ~~ \forall ~ p\in E,~~p\neq p_0, ~~\text{and}~~ \forall ~\gamma_{\2}\in \Gamma_0,$$
		with $\gamma_{\2}(0)=p_0$ and $\dot{\gamma}_{\2}(0)=X_{p_0}\in T_{p_0}(E).$
	\end{enumerate}
\end{theorem}

The following theorem establishes the necessary and sufficient requirement for a point to qualify as a local minimum for a gH-directionally differentiable IVF that is convex at that particular point.

\begin{theorem} \label{theorem5.5,p3}
	Let $E\subseteq (M,g)$ be star-shaped at $p_0 \in E$ and $f:E\rightarrow \mathbb{I}$ be an IVF with $f(p)=\langle f^c(p), f^w(p)\rangle$. Let $\gamma(s);~s\in I, ~0\in I~ \&~ \gamma(I)\subseteq E$, be  geodesic for which $\gamma(0)=p_0, ~ \&~ \dot{\gamma}(0)=X_{p_0}\in T_{p_0}(E)$ such that $(f^w \circ \gamma)(s)$ is non-decreasing for $s\in [0, \infty)$. Suppose that $f$ is gH-directionally differentiable and cw-convex at $p_0$, then $p_0$ is a local minimum point of $f$ if and only if $0\leq^{\text{min}} Df(p_0;X_{p_0}),~ \forall~X_{p_0} \in T_{p_0}(E)$.
\end{theorem}

\begin{proof}
	Let $p_0$ be a local minimum point of $f$. For sufficiently small, $s\geq 0$, we have $f(p_0)\leq^{\text{min}} f(\gamma(s))$, where $\gamma$ is arbitrary geodesic emanating from $p_0$ in any arbitrary direction $X_{p_0}\in T_{p_0}(E)$, which from order relation (\ref{orderrelation,p2}) gives $f^c(p_0)\leq f^c(\gamma(s))$ and hence one has
	\begin{equation}\label{eq11,p3}
		Df^c(p_0;X_{p_0})=\lim\limits_{s\rightarrow0^+}\frac{f^c(\gamma(s))-f^c(p_0)}{s}\geq 0.
	\end{equation}
	Also, by hypothesis, $(f^w\circ \gamma)(s)$ is non-decreasing for $s\geq0$. So, we have
	\begin{equation}\label{eq12,p3}
		Df^w(p_0;X_{p_0})=\lim\limits_{s\rightarrow0^+}\frac{f^w(\gamma(s))-f^w(p_0)}{s}\geq 0.
	\end{equation}
	From expressions (\ref{eq11,p3})and (\ref{eq12,p3}), one can deduce that
	\begin{align*}
		Df^c(p_0;X_{p_0})>0, ~~& \text{when~}Df^c(p_0;X_{p_0}) \neq 0;\\
		Df^w(p_0;X_{p_0})\geq 0, ~~& \text{when } Df^c(p_0;X_{p_0}) = 0.
	\end{align*}
	This from order relation (\ref{orderrelation,p2}), yields~ $0\leq^{\text{min}}Df(p_0;X_{p_0})$. Since $\gamma$ is arbitrary, one has
	$$0\leq^{\text{min}}Df(p_0;X_{p_0}) ~~ \forall ~X_{p_0}\in T_{p_0}(E).$$
	
	Conversely, suppose that $0\leq^{\text{min}}Df(p_0;X_{p_0}) ~~ \forall ~X_{p_0}\in T_{p_0}(E).$ From Theorem \ref{theorem23,p2} and transitivity of total order, we have $0\leq^{\text{min}}f(p)\ominus_{gH}f(p_0)~ \forall~p\in E$. Which from Lemma \ref{lemma13,p2}, gives 
	$$0\leq^{\text{min}} \langle f^c(p)-f^c(p_0), |f^w(p)-f^w(p_0)| \rangle.$$
	From order relation (\ref{orderrelation,p2}), we have
	\begin{equation} \label{eq13,p3}
		0\leq f^c(p)-f^c(p_0), ~~ \forall~p\in E,
	\end{equation}
	Since, for any $p\in E$ and the geodesic $\gamma(s), s \in [0,1],$ joining $p_0=\gamma(0)$ and $p=\gamma(1)$, $(f^w\circ \gamma)(s)$ is non-decreasing for $s>0$. So,
	\begin{equation} \label{eq14,p3}
		f^w(p_0) \leq f^w(p), ~~ \forall ~ p\in E.
	\end{equation}
	From (\ref{eq13,p3}) and (\ref{eq14,p3}), one can deduce that
	\begin{align*}
		f^c(p_0)<f^c(p), ~~& \text{when } f^c(p_0)\neq f^c(p);\\
		~f^w(p_0)\leq f^w(p),~~ & \text{when } f^c(p_0)= f^c(p).
	\end{align*}
	This together with order relation (\ref{orderrelation,p2}), yields
	$$f(p_0)\leq^{\text{min}}f(p) ~~ \forall ~ p \in E.$$
	Thus, $p_0$ is a local minimum of $f(p)$.
	\end{proof}

\section{Optimality conditions for constrained optimization programming problem}
In this section, we present the KKT type optimality conditions for real-valued as well as interval-valued optimization problem on $(M,g)$. We first consider the following real-valued optimization problem on $(M,g)$.
$$\begin{array}{lll}
	(P_2)\hspace{1cm} & \text{minimize} & f(p)\vspace{0.1cm}\hspace{2cm}\\ 
	& \text{subject to} & g_i(p)\leq 0, ~~ i\in \{1,2,...,m\},
\end{array}$$
where $f,g_i:E\rightarrow \mathbb{R},~ i\in \{1,2,...,m\}$, $E\subseteq(M,g)$ and the set $X=\{p\in E:g_i(p)\leq 0,~ 1\leq i \leq m\}$ is the feasible set.

\begin{definition}
	\rm A function $f:E\rightarrow \mathbb{R}$, defined on a subset $E\subseteq (M,g)$, is said to be non-constant on $E$ if for any $p,q \in E, ~ p\neq q,$ we have $f(p)\neq f(q)$.
\end{definition}

\begin{definition}
	\rm A IVF function $f:E\rightarrow \mathbb{I}$, defined on a subset $E\subseteq (M,g)$, is said to be non-constant on $E$ if for any $p,q \in E, ~ p\neq q,$ we have $f(p)\neq f(q)$.
\end{definition}
\vspace{0.15cm}
The following theorem presents the conditions that are sufficient for $p_0$ to be an optimal solution for $(P_2)$.

\begin{theorem} \label{theoremhash1,p3}
	Let $E\subseteq (M,g)$ be star-shaped at $p_0\in E$. Let $p_0 \in X$  be a feasible point and $J=\{i:g_i(p_0)=0\}$. Suppose that the objective function $f$ and the constraints $g_i,~i\in J$, are convex at $p_0$, and $f$ and $g_i,~ i\in \{1,2,...,m\}$, are directionally differentiable at $p_0$. If there exist scalars $0\leq \mu_i \in \mathbb{R},~ i \in \{1,2,...,m\}$, such that
	\begin{enumerate}
		\item[(i)] $Df(p_0;X_{p_0}) + \displaystyle\sum_{i=1}^{m}\mu_i Dg_i(p_0;X_{p_0}) \geq 0;$
		\item [(ii)] $\mu_i g_i(p_0)=0 ~~ \forall~ i \in \{1,2,...,m\}.$
	\end{enumerate}
	Then, $p_0$ is an optimal solution to Problem $(P_2)$. Moreover, if the objective function $f$ is non-constant on $X$, then $p_0$ is strict optimal solution to Problem $(P_2)$.
\end{theorem}
\begin{proof}
	For any feasible solution $p \in X$, we have from $g_i(p)\leq 0, ~ \forall ~i \in \{1,2,...,m\}$, that 
	\begin{equation}\label{eq15,p3}
		g_i(p)\leq g_i(p_0) =0, ~~ \forall~ i \in J.
	\end{equation}
	Let $\gamma_{\2}(s), ~s\in [0,1]$, be a geodesic such that  $\gamma_{\2}(0)=p_0, ~\gamma_{\2}(1)=p$ and $ \dot{\gamma}_{\2}(0)=X_{p_0}$. From (\ref{eq15,p3}) and Theorem \ref{theorem30,p2}, we have
	\begin{equation}\label{eq16,p3}
		0\geq g_i(p)-g_i(p_0) \geq Dg_i(p_0;X_{p_0}), ~~ \forall ~i \in J.
	\end{equation}
	For $\mu_i\geq 0,~ i \in J$, (\ref{eq16,p3}) yields that
	$$\displaystyle \sum_{i\in J} \mu_iDg_i(p_0;X_{p_0}) \leq 0.$$
	This together with $\mu_ig_i(p_0)=0$, gives
	\begin{equation}\label{eq17,p3}
		\displaystyle \sum_{i=1}^{m} \mu_iDg_i(p_0;X_{p_0}) \leq 0, ~~~ (\mu_i=0,~ \forall~i \notin J).
	\end{equation}
	From $(i)$ and (\ref{eq17,p3}), we have
	$$Df(p_0;X_{p_0})\geq -\displaystyle \sum_{i=1}^{m} \mu_iDg_i(p_0;X_{p_0}) \geq 0.$$
	This from convexity of $f$ at $p_0$, yields
	$$f(p)-f(p_0) \geq Df(p_0;X_{p_0}) \geq 0, ~~\forall ~p \in X,$$
	$$\text{i.e.}, ~~~f(p)\geq f(p_0), ~~ \forall ~ p\in X.$$
	Thus, $p_0$ is an optimal solution to Problem $(P_2)$.
	
	Moreover, if $f$ is non-constant on $X$, then $f(p)\neq f(p_0)$ for any $p\in X$. So, $f(p)>f(p_0) ~ \forall ~ p\in X$. Hence, $p_0$ is strict optimal solution to Problem $(P_2)$.
	\end{proof}
The following example is in support of Theorem \ref{theoremhash1,p3}.
\begin{example}
	\rm Consider the Riemannian manifold \( \mathit{M} \) as defined in Example \ref{example6*,p2}. Let $\mathit{E}=\left\{e^{i\theta}: \theta\in [0,\pi]\right\}$ be a subset of $\mathit{M}$ which is star-shaped $p_0 = e^{i\frac{\pi}{2}}=i$. The geodesic segment joining $p_0=i$ with any $p=e^{i\theta}\in E$ is given by
	$$\gamma_{\2}(s) = e^{i(\frac{\pi}{2}+s(\theta - \frac{\pi}{2}))}, ~~~ s\in [0,1],$$	
	and the  tangent vector to the geodesic is  $X_{p_0}=\dot{\gamma}_{\2}(0)=\frac{\pi}{2}-\theta \in T_{p_0}(E) \subset \mathbb{R}.$
	
	Let $f: M \rightarrow \mathbb{R}$ be the real valued objective function and $g_i:M \rightarrow \mathbb{R}, ~i\in \{1,2,3\},$ be the real valued constraints for the following optimization problem,
	$$\begin{array}{lllll}
		(P^*)~~~~~ & \text{minimize} & f(p)&\hspace{-0cm}=f(e^{i\theta})&\hspace{-0cm}=\frac{\pi}{2} - \theta \vspace{0cm}\hspace{2cm}\vspace{0.2cm}\\ 
		& \text{subject to} & g_1(p)&\hspace{-0cm}=g_1(e^{i\theta}) &\hspace{-0cm}= \theta - \frac{\pi}{2}\leq 0,\vspace{0.1cm}\\
		& & g_2(p)&\hspace{-0cm}=g_2(e^{i\theta}) &\hspace{0cm}= e^{(\theta-\frac{\pi}{2})}-1\leq 0,\vspace{0.1cm}\\
		& & g_3(p)&\hspace{-0cm}=g_3(e^{i\theta}) &\hspace{-0cm}= -\ln(9\pi^2-(\theta - \frac{\pi}{2})^2)\leq 0.
	\end{array}$$
	Here, the feasible region is $\mathcal{X}=\{e^{i\theta} : \theta \in [0, \frac{\pi}{2}]\}$ and the set $J=\{i:g_i(p_0)=0\}=\{1,2\}.$ In this case, both the objective function and all the constraints are convex at $p_0$.
	
	The directional derivatives of the functions involved in $(P^*)$ at $p_0=i$ in the direction $X_{p_0}= \frac{\pi}{2}-\theta$ are given by
	\begin{align*}
		Df(i; \frac{\pi}{2}-\theta)&=\frac{\pi}{2}-\theta;\\
		Dg_1(i; \frac{\pi}{2}-\theta)&=\theta -\frac{\pi}{2};\\
		Dg_2(i; \frac{\pi}{2}-\theta)&=\theta -\frac{\pi}{2};\\
		Dg_3(i; \frac{\pi}{2}-\theta)&=0.
	\end{align*}
	It is easy to see that the conditions {\it (i)} and {\it (ii)} of Theorem \ref{theoremhash1,p3} hold at $p_0=i$ with $(\mu_1,\mu_2,\mu_3)=(\frac{1}{2},\frac{1}{2},0)$. So, by Theorem \ref{theoremhash1,p3}, we conclude that $p_0=i$ is the optimal solution to the Problem $(P^*)$. The optimal value is $0$. Moreover, the objective function $f$ is non-constant on $\mathcal{X}$, we have that $p_0$ is strict optimal solution to Problem $(P^*)$.
\end{example}

Next, we consider an IVOP problem on $(M,g)$ in which the objective function is interval-valued and constraints are real valued, as following
$$\begin{array}{lll}
	(P_3)\hspace{1cm} & \text{minimize} & f(p)=\langle f^c(p), f^w(p) \rangle \vspace{0.1cm}\hspace{2cm}\\ 
	& \text{subject to} & g_i(p)\leq 0, ~~ i\in \{1,2,...,m\},
\end{array}$$
where $f:E \rightarrow \mathbb{I}~,g_i:E\rightarrow \mathbb{R},~ i\in \{1,2,...,m\}$, $E\subseteq(M,g)$ and the set $X=\{p\in E:g_i(p)\leq 0,~ 1\leq i \leq m\}$ is the feasible set.\vspace{0.15cm}

The following theorem presents the conditions that are sufficient for $p_0$ to be an optimal solution for $(P_3)$.

\begin{theorem}
	\label{theoremhash2,p3}
	Let $E\subseteq (M,g)$ be star-shaped at $p_0\in E$. Let $p_0 \in X$  be a feasible point and $J=\{i:g_i(p_0)=0\}$. Suppose that the interval-valued objective function $f$ is cw-convex at $p_0$ and the real-valued constraints $g_i,~i\in J$, are convex at $p_0$. Let $\gamma(s),~s\in I,~0\in I,~\&~\gamma(I)\subseteq E,$ be any geodesic such that $\gamma(0)=p_0$, $\dot{\gamma}(0)=X_{p_0}$ and $(f^w\circ \gamma)(s)$ is non-decreasing for $s\in I\cap [0, \infty)$. Suppose that $f$ is gH-directionally differentiable at $p_0$ and $g_i,~ i\in \{1,2,...,m\}$, are directionally differentiable at $p_0$. If there exist scalars $0\leq \mu_i \in \mathbb{R},~ i \in \{1,2,...,m\}$, such that
	\begin{enumerate}
		\item[(i)]  $0 \leq^{\text{min}} Df(p_0;X_{p_0}) + \displaystyle\sum_{i=1}^{m}\mu_i Dg_i(p_0;X_{p_0});$
		\item [(ii)] $\mu_i g_i(p_0)=0 ~~ \forall~ i \in \{1,2,...,m\}.$
	\end{enumerate}
	Then, $p_0$ is an optimal solution to Problem $(P_3)$. Moreover, if the objective function $f$ is non-constant on $X$, then $p_0$ is strict optimal solution to Problem $(P_3)$.
\end{theorem}
\begin{proof}
	From Lemma \ref{lemma2,p2}{\bf (iii)} and condition {\it (i)}, we have
	\begin{equation}\label{eq18,p3}
		-\displaystyle\sum_{i=1}^{m}\mu_i Dg_i(p_0;X_{p_0})\leq^{\text{min}} Df(p_0;X_{p_0}).
	\end{equation}
	By hypothesis of Theorem \ref{theoremhash1,p3} and Theorem \ref{theoremhash2,p3}, we can directly use inequality (\ref{eq17,p3}) which together with inequality (\ref{eq18,p3}) and transitivity of total order, gives $0 \leq^{\text{min}} Df(p_0;X_{p_0})$. From Theorem \ref{theorem5.5,p3}, it yields that $p_0$ is an optimal solution to Problem $(P_3)$.
	\end{proof}

Theorem \ref{theoremhash2,p3} provides the sufficient conditions for $(P_3)$ to have an optimal solution. For this, the domain of the objective function and the constraints within $(P_3)$ is required to be star-shaped at a particular point. As we know, the concept of a star-shaped set encompasses broader scope than that of a convex set. Moreover, the condition for the domain being star-shaped at a particular point allows one to discuss the convexity of the functions at that point. Consequently, the KKT-type conditions furnished by Theorem \ref{theoremhash2,p3} possess a wider applicability compared to those proposed by the authors in references such as \cite{chalco-cano,chen,sadikur,bilal1,bilal2,wu1,wu2}. To illustrate the validity of Theorem \ref{theoremhash2,p3}, we consider an example which follows next. This example cannot be effectively addressed using the methodologies advanced by the aforementioned authors. However, it becomes solvable by leveraging the insights provided by Theorem \ref{theoremhash2,p3}.
\begin{example}
	\rm The collection $S^2_{++}$ of $2 \times 2$ symmetric positive definite matrices with entries from $\mathbb{R}$ is a Riemannian manifold with Riemannian metric:
	$$g_p(X,Y) = Tr(p^{-1}Xp^{-1}Y), ~~~ \forall~ p \in S^2_{++}, ~~~ X,Y \in T_p(S^2_{++}).$$
	The unique minimal geodesic joining $p, q \in S^2_{++}$ is given by 
	$$\gamma_{\1}(s) = p^{\frac{1}{2}}(p^{-\frac{1}{2}}qp^{-\frac{1}{2}})^sp^{\frac{1}{2}}, ~~~ \forall ~ s \in [0,1].$$
	For more details, one can refer to \cite{bacak,serge}.\\
	
	\noindent Let $p_0=I$, $p_1=2I$ and $p_2=\begin{pmatrix}
		1&0\\0&2
	\end{pmatrix}$ be three matrices in $S^2_{++}$, where $I$ is the identity matrix. Then the geodesics emanating from $p_0$ to $p_1$ and $p_2$ are respectively given by:
	$$\gamma_{\7}(s)=2^sI, ~s\in[0,1] ~~\text{and}~~ \gamma_{\8}(s)=p_2^s, ~s\in[0,1].$$
	Suppose that,
	$$E=\left\{p\in S^2_{++}:p=\gamma_{\7}(s)~\text{or}~ p=\gamma_{\8}(s), s \in [0,1]\right\}.$$
	Clearly, $E$ is star-shaped at $p_0$.\\
	
	\noindent We consider the following Interval-valued optimization problem on $E$,
	$$\begin{array}{lllll}
		(P^{**})~~~~	& \text{minimize} & f(p)=\langle f^c(p),f^w(p)\rangle; \vspace{0.1cm}\hspace{2cm}\\ 
		& \text{subject to} & g_i(p)\leq 0, ~~~ i \in \{1,2,3\}.
	\end{array}$$
	where, $f: E \rightarrow\mathbb{I}$ is defined for any $p\in E$ as follows:
	$$f(p) = \begin{cases}
		\left\langle \ln(\det(p)),1\right\rangle; & p=\gamma_{\7}(s),s\in[0,1],\\
		\langle 0,~ 1 \rangle; & p=\gamma_{\8}(s),s\in[0,1],\\
	\end{cases}$$
	$g_i: E \rightarrow\mathbb{R}, ~i\in \{1,2,3\}$ are defined for any $p\in S^2_{++}$, as follows:
	\begin{align*}
		g_1(p) &= \begin{cases}
			-\ln(\det(p)); & p=\gamma_{\7}(s),s\in[0,1],\\
			0; & p=\gamma_{\8}(s),s\in[0,1],\\
		\end{cases}\\
		g_2(p) &= \begin{cases}
			-\left(\ln(\det(p))\right)^2-1; & p=\gamma_{\7}(s),s\in[0,1],\\
			-1; & p=\gamma_{\8}(s),s\in[0,1],
		\end{cases}\\
		g_3(p) &= \begin{cases}
			\ln(\det(p))-1; & p=\gamma_{\7}(s),s\in[0,1],\\
			-1; & p=\gamma_{\8}(s),s\in[0,1].
		\end{cases}
	\end{align*}
	Here, the feasible region is:
	$$X=\left\{p=\gamma_{\7}(s):s\in[0,1], 1\leq \det(p)\leq e\right\}\bigcup \left\{p=\gamma_{\8}(s):s\in[0,1]\right\}.$$
	Also, for any $\gamma(s)$ emanating from $p_0$ to any $p\in E$, we have 
	$$(f^w \circ \gamma)(s) = 1,~s\in[0,1].$$
	which is non-decreasing in $s\geq 0$. Also, one can see $f$ is cw-convex at $p_0$. The set $j=\{i:g_i(p_0)=0\}=\{1\}$, and $g_1$ and $g_3$ are convex at $p_0$ but $g_2$ fails to be convex at $p_0$.\\
	
	\noindent The gH-directional derivative of $f$ at $p_0$ in any direction $X_{p_0}=\exp_{p_0}^{-1}q\in T_{p_0}(E)$ is given by
	$$Df(p_0;X_{p_0})= \begin{cases}
		\langle \ln(det(q)),0 \rangle; ~~ q=\gamma_{\7}(s),s\in[0,1],\\
		\langle 0,0 \rangle; ~~ q=\gamma_{\8}(s),s\in[0,1].
	\end{cases}$$
	The directional derivatives of $g_1,g_2$ and $g_3$ at $p_0$ in any direction $X_{p_0}=\exp^{-1}_{p_0}q\in T_{p_0}(E)$ are given by
	\begin{align*}
		Dg_1(p_0;X_{p_0})&= \begin{cases}
			-\ln(\det(q)); ~~ q=\gamma_{\7}(s),s\in[0,1],\\
			0; ~~ q=\gamma_{\8}(s),s\in[0,1],
		\end{cases}\\
		Dg_2(p_0;X_{p_0})&= \begin{cases}
			0; ~~ q=\gamma_{\7}(s),s\in[0,1],\\
			0; ~~ q=\gamma_{\8}(s),s\in[0,1],
		\end{cases}\\
		Dg_3(p_0;X_{p_0})&= \begin{cases}
			\ln(\det(p)); ~~ q=\gamma_{\7}(s),s\in[0,1],\\
			0; ~~ q=\gamma_{\8}(s),s\in[0,1].
		\end{cases}
	\end{align*}
	One can check for $(\mu_1,\mu_2,\mu_3)=(1,0,0)$ that conditions {\it (i)} and {\it (ii)} in Theorem \ref{theoremhash2,p3} hold true at $p_0$. Hence, from Theorem \ref{theoremhash2,p3}, we conclude that $p_0$ is an optimal solution to IVOP $(P^{**})$. The optimal value is $\langle 0,1 \rangle.$ Moreover, the objective function $f$ is not non-constant on $X$, we have that $p_0$ is not a strict optimal solution to Problem $(P^{**})$.
\end{example}

In the next theorem, we present conditions that are sufficient for $p_0$ to be an optimal solution to $(P_3)$ in terms of center function $f^c$ and half-width function $f^w$.

\begin{theorem}
	\label{theoremhash3,p3}
	Let $E\subseteq (M,g)$ be star-shaped at $p_0\in E$. Let $p_0 \in X$  be a feasible point and $J=\{i:g_i(p_0)=0\}$. Suppose that the real-valued constraints $g_i,~i\in J$, are convex at $p_0$ and  $g_i,~ i\in \{1,2,...,m\}$, are directionally differentiable at $p_0$. Then we have the following;
	\begin{enumerate}
		\item Suppose that the center function $f^c$ of the objective function $f$ is directionally differentiable at $p_0$, convex at $p_0$ and non-constant on $X$. If there exist scalars $0\leq \mu_i \in \mathbb{R},~ i \in \{1,2,...,m\}$, such that
		\begin{enumerate}
			\item[(i)]  $Df^c(p_0;X_{p_0}) + \displaystyle\sum_{i=1}^{m}\mu_i Dg_i(p_0;X_{p_0})\geq 0;$
			\item [(ii)] $\mu_i g_i(p_0)=0 ~~ \forall~ i \in \{1,2,...,m\}.$
		\end{enumerate}
		Then, $p_0$ is a strict optimal solution to Problem $(P_3)$;
		\item Suppose that the center function $f^c$ of the objective function $f$ is constant on $X$ and the half-width function $f^w$ of the objective function $f$ is directionally differentiable at $p_0$ and convex at $p_0$. If there exist scalars $0\leq \mu_i \in \mathbb{R},~ i \in \{1,2,...,m\}$, such that
		\begin{enumerate}
			\item[(i)]  $Df^w(p_0;X_{p_0}) + \displaystyle\sum_{i=1}^{m}\mu_i Dg_i(p_0;X_{p_0})\geq 0;$
			\item [(ii)] $\mu_i g_i(p_0)=0 ~~ \forall~ i \in \{1,2,...,m\}.$
		\end{enumerate}
		Then, $p_0$ is an optimal solution to Problem $(P_3)$. Moreover, if $f^w$ is non-constant on $X$, then $p_0$ is strict optimal solution to Problem $(P_3)$.
	\end{enumerate}
\end{theorem}
\begin{proof}
	{\it 1.} From Theorem \ref{theoremhash1,p3}, $p_0$ is a strict optimal solution to the following real-valued convex problem.
	$$\begin{array}{lll}
		& \text{Minimize} & f^c(p),\vspace{0.1cm}\hspace{2cm}\\ 
		& \text{subject to,} & g_i(p)\leq 0, ~~ i\in \{1,2,...,m\},
	\end{array}$$
	i.e., $f^c(p_0) < f(p),~\forall~p\in X,$ which from order relation (\ref{orderrelation,p2}) yields that $p_0$ is strict optimal solution to Problem $(P_3)$.
	
	The proof of Part {\it 2} is similar to that of Part {\it 1}.
	\end{proof}

We now consider an IVOP problem on $(M,g)$ in which both objective and constraint functions are interval-valued, as follows:
$$\begin{array}{lll}
	(P_4)\hspace{1cm} & \text{minimize} & f(p)=\langle f^c(p), f^w(p) \rangle \vspace{0.1cm}\hspace{2cm}\\ 
	& \text{subject to} & g_i(p)=\langle g_i^c(p), g_i^w(p)\rangle \leq^{\text{min}} 0, ~~ i\in \{1,2,...,m\},
\end{array}$$
where $f,g_i:E \rightarrow \mathbb{I}~ i\in \{1,2,...,m\}$, $E\subseteq(M,g)$ and the set $X=\{p\in E:g_i(p)\leq 0,~ 1\leq i \leq m\}$ is the feasible set.\vspace{0.15cm}

The following theorem provides conditions that are the sufficient for $p_0$ to be an optimal solution to $(P_4)$.

\begin{theorem}
	\label{theoremhash4,p3}
	Let $E\subseteq (M,g)$ be star-shaped at $p_0\in E$. Let $p_0 \in X$  be a feasible point and $J=\{i:g_i(p_0)=0\}$. Let $\gamma(s),~s\in I,~0\in I,~\&~\gamma(I)\subseteq E,$ be any geodesic such that $\gamma(0)=p_0$, $\dot{\gamma}(0)=X_{p_0}$ and $(g_i^w\circ \gamma)(s), ~i\in \{1,2,...m\},$ is non-decreasing for $s\in I\cap [0, \infty)$. Suppose that the interval-valued constraints $g_i,~i\in J$, are cw-convex at $p_0$ and  $g_i,~ i\in \{1,2,...,m\}$, are gH-directionally differentiable at $p_0$. Then, we have the following;
	\begin{enumerate}
		\item Suppose that the center function $f^c$ of the objective function $f$ is directionally differentiable at $p_0$, convex at $p_0$ and non-constant on $X$. If there exist scalars $0\leq \mu_i \in \mathbb{R},~ i \in \{1,2,...,m\}$, such that
		\begin{enumerate}
			\item[(i)]  $0\leq^{\text{min}}Df^c(p_0;X_{p_0}) + \displaystyle\sum_{i=1}^{m}\mu_i Dg_i(p_0;X_{p_0});$
			\item [(ii)] $\mu_i g_i(p_0)=0 ~~ \forall~ i \in \{1,2,...,m\}.$
		\end{enumerate}
		Then, $p_0$ is strict optimal solution to Problem $(P_4)$;
		\item Suppose that the center function $f^c$ of the objective function $f$ is constant on $X$ and the half-width function $f^w$ of the objective function $f$ is directionally differentiable at $p_0$ and convex at $p_0$. If there exist scalars $0\leq \mu_i \in \mathbb{R},~ i \in \{1,2,...,m\}$, such that
		\begin{enumerate}
			\item[(i)]  $0\leq^{\text{min}}Df^w(p_0;X_{p_0}) + \displaystyle\sum_{i=1}^{m}\mu_i Dg_i(p_0;X_{p_0});$
			\item [(ii)] $\mu_i g_i(p_0)=0 ~~ \forall~ i \in \{1,2,...,m\}.$
		\end{enumerate}
		Then, $p_0$ is optimal solution to Problem $(P_4)$. Moreover, if $f^w$ is non-constant on $X$, then $p_0$ is strict optimal solution to Problem $(P_4)$.
	\end{enumerate}
\end{theorem}
\begin{proof}
	{\it 1.} For any feasible point $p\in X$, we have, from $g_i(p)\leq^{\text{min}}0$, $i\in \{1,2,...,m\},$ and $g_i(p_0)=0,~ i\in J,$ that
	$$g_i(p) \ominus_{gH} g_i(p_0) \leq^{\text{min}}0, ~~ \forall~i\in J.$$
	Since, $g_i(p_0), i\in J$, are cw-convex at $p_0$, we have from Theorem \ref{theorem23,p2} and transitivity of total order, that
	$$Dg_i(p_0;X_{p_0}) \leq^{\text{min}} 0.$$
	For $0\leq \mu_i\in \mathbb{R}, ~ i\in J$, and using induction on Parts {\bf (i)} and {\bf (ii)} of Lemma \ref{lemma2,p2}, we have
	$$\displaystyle \sum_{i\in J}\mu_iDg_i(p_0;X_{p_0}) \leq^{\text{min}} 0,$$
	which together with $\mu_ig_i(p_0)=0$, yields
	\begin{equation} \label{eq19,p3}
		\sum_{i=1}^{m}\mu_iDg_i(p_0;X_{p_0}) \leq^{\text{min}} 0.
	\end{equation}
	From condition {\it 1.(i)} and Lemma \ref{lemma2,p2}{\bf (iii)}, we have
	\begin{equation}
		\label{eq20,p3} -Df^c(p_0;X_{p_0}) \leq^{\text{min}} \sum_{i=1}^{m}\mu_iDg_i(p_0;X_{p_0}).
	\end{equation}
	Inequalities (\ref{eq19,p3}) and (\ref{eq20,p3}), together yield that $Df^c(p_0;X_{p_0}) \geq 0$. Which from Theorem \ref{theorem30,p2} gives $f^c(p_0)\leq f^c(p), ~\forall~p\in X$. But $f^c$ is non-constant on $X$, we have $f^c(p_0) < f^c(p), ~\forall~p\in X$. This by order relation (\ref{orderrelation,p2}) yields that $f(p_0)<^{\text{min}}f(p),~\forall ~ p\in X$. Thus, $p_0$ is strict optimal solution to Problem $(P_4)$.\\
	
	\noindent {\it 2.} Its proof is similar to that of Part {\it 1.}
	\end{proof}

From Problem $(P_4)$, we have the condition $g_i(p)\leq^{\text{min}}0$, which, from order relation (\ref{orderrelation,p2}), yields: 
$$g_i^c(p)\leq 0 ~~\text{and}~~ g_i^w(p)\geq 0, ~~ \forall~ i\in \{1,2,...,m\}.$$ 
Using the last two expressions, we define the following problem:
$$\begin{array}{lll}
	(P_5)\hspace{1cm} & \text{minimize} & f(p)=\langle f^c(p), f^w(p) \rangle \vspace{0.1cm}\hspace{2cm}\\ 
	& \text{subject to} & g_i^c(p)\leq0,  ~~ i\in \{1,2,...,m\},\\
	& &  -g_i^w(p)\leq 0, ~~ i\in \{1,2,...,m\}.
\end{array}$$
It is important to note that problem $(P_5)$ is equivalent to problem $(P_3)$. Additionally, it can be shown that the feasible set $\mathcal{X}^4$ of problem $(P_4)$ is a subset of the feasible set $\mathcal{X}^3$ of problem $(P_3)$. This implies that if $p^* \in \mathcal{X}^4$ is an optimal solution to problem $(P_3)$, then $p^*$ is also an optimal solution to problem $(P_4)$. Therefore, the problem $(P_4)$ can also be addressed by the results established in this article for the problem $(P_3)$, provided that the optimal solution lies within the feasible set $\mathcal{X}^4$ of the problem $(P_4)$.

\section{Conclusion}
By utilizing gH-directional differentiability, we have successfully obtained KKT-type optimality conditions for an interval-valued optimization problem on Riemannian manifolds in this article. By an example, we have demonstrated the superiority of these KKT conditions over those developed in Euclidean spaces, as evidenced by previous studies \cite{chalco-cano,chen,sadikur,bilal1,bilal2,wu1,wu2}. These KKT conditions have vast potential for applications in machine learning and artificial intelligence, and one can explore the development of optimization techniques and algorithms to determine the optimal solution in the future. Additional inquiries into nonlinear optimization, such as examining saddle point criteria, constraint qualifications, duality theory, and other related topics, can be conducted on spaces that are nonlinear in nature.


\begin{thebibliography}{10}
		\expandafter\ifx\csname url\endcsname\relax
	\def\url#1{\texttt{#1}}\fi
	\expandafter\ifx\csname urlprefix\endcsname\relax\def\urlprefix{URL }\fi
	\expandafter\ifx\csname href\endcsname\relax
	\def\href#1#2{#2} \def\path#1{#1}\fi
	
	\bibitem{alefeld}
	{G. Alefeld, J. Herzberger}.
	\newblock {\em {Introduction to interval computations}}.
	\newblock {NY: Academic Press}, {1983}.
	
	\bibitem{bacak}
	{M. Bacak}.
	\newblock {\em {Convex Analysis and Optimization in Hadamard Spaces}}.
	\newblock {De Gruyter Series in Nonlinear Analysis and Applications}, {22,
		2014}.
	\newblock \href{https://doi.org/10.1515/9783110361629}{10.1515/9783110361629}.
	
	\bibitem{bento2}
	{G.C. Bento, J.G. Melo}.
	\newblock {Subgradient method for convex feasibility on Riemanian manifolds}.
	\newblock {\em {Journal of Optimization Theory and Applications}},
	{152}:{773--785}, {2012}.
	\newblock \href{https://doi.org/10.1007/s10957-011-9921-4}{10.1007/s10957-011-9921-4}.
	
	\bibitem{bento1}
	{G.C. Bento, O.P. Ferreira, P.R. Oliveira}.
	\newblock {Proximal point method for a special class of nonconvex functions on
		Hadamard manifolds}.
	\newblock {\em {Optimization: Taylor \& Francis online}}, {64}:{289--319},
	{2015}.
		\newblock \href{https://doi.org/10.1080/02331934.2012.745531}{10.1080/02331934.2012.745531}.
	
	
	\bibitem{hilal}
	{H.A. Bhat and A.Iqbal}.
	\newblock {Generalized Hukuhara directional differentiability of
		interval-valued functions on Riemannian manifolds}.
	\newblock {\em {arXiv}}, {2023}.
	\newblock \href {https://doi.org/10.48550/arXiv.2212.04541}
	{\path{doi:10.48550/arXiv.2212.04541}}.	
	
	\bibitem{bhunia}
	{A.K. Bhunia, S.S. Samanta}.
	\newblock {A study of interval metric and its application in multi-objective
		optimization with interval objectives}.
	\newblock {\em {Computers \& Industrial Engineering}}, 74:169--178, 2014.
		\newblock \href {https://doi.org/10.1016/j.cie.2014.05.014}
	{\path{doi:10.1016/j.cie.2014.05.014}}.	
	
	\bibitem{chalco-cano}
	{Y. Chalco-Cano, H. Román-Flores, M.D. Jiménez-Gamero}.
	\newblock {Generalized derivative and $\pi$-derivative for set-valued
		functions}.
	\newblock {\em {Information Sciences}}, {11(181) (2011) (2177-2188)}.
	\newblock \href {https://doi.org/10.1016/j.ins.2011.01.023}
	{\path{doi:10.1016/j.ins.2011.01.023}}.	
	
	\bibitem{chen}
	{S.-L. Chen}.
	\newblock {The KKT optimality conditions for optimization problem with
		interval-valued objective function on Hadamard manifolds}.
	\newblock {\em {Optimization, Taylor \& Francis}}, {2020}.
	\newblock \href {https://doi.org/10.1080/02331934.2020.1810248}
	{\path{doi:10.1080/02331934.2020.1810248}}.
	
	\bibitem{ferreira}
	{O.P. Ferreira, L.L.R. Perez, S.Z. Nemeth}.
	\newblock {Singularities of monotone vector fields and an extragradient-type
		algorithm}.
	\newblock {\em {Journal of Global Optimization}}, {31}:{133–151}, {2005}.
	\newblock \href {https://doi.org/10.1007/s10898-003-3780-y}
	{\path{doi:10.1007/s10898-003-3780-y}}.
	
	\bibitem{gabriel}
	{R.-G. Gabriel, O.-G. Rafaela, R.-L. Antonio}.
	\newblock {Optimality and duality on Riemannian manifolds}.
	\newblock {\em {Taiwanese Journal of Mathematics, Mathematical Society of the
			Republic of China}}, {22}:{1245--59}, {2018}.
		\newline\urlprefix\url{www.jstor.org/stable/90025383}.
	
	\bibitem{ishibuchi}
	{H. Ishibuchi, H. Tanaka}.
	\newblock {Multiobjective programming in optimization of the interval objective
		function}.
	\newblock {\em {European Journal of Operational Research}}, {48}:{219--225},
	{1990}.
	\newblock \href {https://doi.org/10.1016/0377-2217(90)90375-L}
	{\path{doi:10.1016/0377-2217(90)90375-L}}.
	
	\bibitem{li}
	{C. Li, J.C. Yao}.
	\newblock {Weak sharp minima on Riemanian manifolds}.
	\newblock {\em {SIAM Journal on Optimization}}, {21}:{1523--1560}, {2011}.
	\newblock \href {https://doi.org/10.1137/09075367X}
	{\path{doi:10.1137/09075367X}}.
	
	\bibitem{moore2}
	{R.E. Moore}.
	\newblock {\em {Interval Analysis}}, volume 158.
	\newblock {Englewood Cliffs (NJ): Prentice-Hall}, 1966.
	\newblock \href {https://doi.org/10.1126/science.158.3799.365}
	{\path{doi:10.1126/science.158.3799.365}}.
	
	\bibitem{moore1}
	{R.E. Moore}.
	\newblock {\em {Method and applications of interval analysis}}.
	\newblock {Philadelphia, SIAM}, {1979}.
	\newblock \href {https://doi.org/10.1137/1.9781611970906}
	{\path{doi:10.1137/1.9781611970906}}.
	
	\bibitem{nemeth}
	{ S.Z. Nemeth}.
	\newblock {Five kinds of monotone vector fields}.
	\newblock {\em {Pure Mathematics and Applications, Department of Mathematics,
			Corvinus University of Budapest}}, {9}:{417--428}, {1998}.
		
	
	\bibitem{sadikur}
	{S.M. Rahman, A.A. Shaikh and A.K. Bhunia}.
	\newblock {Necessary and sufficient optimality conditions for non-linear
		unconstrained and constrained optimization problem with interval valued
		objective function}.
	\newblock {\em {Computers \& Industrial Engineering}}, 147:106634, 2020.
	\newblock \href {https://doi.org/10.1016/j.cie.2020.106634}
	{\path{doi:10.1016/j.cie.2020.106634}}.
	
	\bibitem{rapcsak1}
	{T. Rapcsak}.
	\newblock {Geodesic convexity in nonlinear optimization}.
	\newblock {\em {Journal of Optimization Theory and Applications}},
	{69}:{169--183}, {1991}.
		\newblock \href {https://doi.org/10.1007/BF00940467}
	{\path{doi:10.1007/BF00940467}}.
	
	\bibitem{rapcsak2}
	{T. Rapcsak}.
	\newblock {\em {Smooth nonlinear optimization in {$\mathbb{R}^n$}, Nonconvex
			Optimization and Its Applications}}, volume~{19}.
	\newblock {Springer-Science + Business Media, B.V.}, {1997}.
		\newblock \href {https://doi.org/10.1007/978-1-4615-6357-0}
	{\path{doi:10.1007/978-1-4615-6357-0}}.
	
	\bibitem{sakai}
	{T. Sakai}.
	\newblock {\em {Riemannian Geometry}}, volume {149}.
	\newblock {Translations of Mathematical Monographs, AMS}, {1996}.
		\newblock \href {https://doi.org/10.1090/mmono/149}
	{\path{doi:10.1090/mmono/149}}.
	
	
	\bibitem{serge}
	{L. Serge}.
	\newblock {\em {Fundamentals of Differential Geometry}}.
	\newblock {Springer New York, NY}, {1999}.
	\newblock \href {https://doi.org/10.1007/978-1-4612-0541-8}
	{\path{doi:10.1007/978-1-4612-0541-8}}.
	
	
	\bibitem{bilal1}
	{D. Singh, B.A. Dar and A. Goyall}.
	\newblock {KKT optimality conditions for interval valued optimization
		problems}.
	\newblock {\em {Journal of Nonlinear Analysis and Optimization}}, 5:91--103, 01
	2014.
	
	\bibitem{bilal2}
	{D. Singh, B.A. Dar and D.S. Kim}.
	\newblock {KKT optimality conditions in interval valued multiobjective
		programming with generalized differentiable functions}.
	\newblock {\em {European Journal of Operational Research}}, 254(1):29--39,
	2016.
	\newblock \href {https://doi.org/10.1016/j.ejor.2016.03.042}
	{\path{doi:10.1016/j.ejor.2016.03.042}}.
	
		
	\bibitem{stefanini}
	{L. Stefanini, B. Bede}.
	\newblock {Generalised Hukuhara differentiability of interval-valued functions
		and interval differential equations}.
	\newblock {\em {Nonlinear Analysis: Theory, Methods \& Applications}},
	71:1311--1328, 2009.
	\newblock \href {https://doi.org/10.1016/j.na.2008.12.005}
	{\path{doi:10.1016/j.na.2008.12.005}}.
	
	\bibitem{loring}
	{L.W. Tu}.
	\newblock {\em {Differential Geometry: Connections, Curvature, and
			Characteristic Classes}}.
	\newblock {Springer Cham}, {2017}.
	\newblock \href {https://doi.org/10.1007/978-3-319-55084-8}
	{\path{doi:10.1007/978-3-319-55084-8}}.
	
	\bibitem{udriste}
	{C. Udriste}.
	\newblock {\em {Convex functions and optimization methods on Riemanian
			manifolds, Mathematics and its Applications}}, volume {297}.
	\newblock {Dordrecht: Kluwer Academy Publishers}, {1994}.
	
	\bibitem{wang}
	{J.M. Wang, G. Lopez, et al.}
	\newblock {Monotone and accretive vector fields on Riemannian manifolds}.
	\newblock {\em {Journal of Optimization Theory and Applications}},
	{146}:{691--708}, {2010}.
	\newblock \href {https://doi.org/10.1007/s10957-010-9688-z}
	{\path{doi:10.1007/s10957-010-9688-z}}.
	
	\bibitem{wu1}
	{H.C. Wu}.
	\newblock {The Karush-Kuhn-Tucker optimality conditions in an optimization
		problem with interval valued objective function}.
	\newblock {\em { European Journal of Operational Research}}, {176}:{46--59},
	{2007}.
	\newblock \href {https://doi.org/10.1016/j.ejor.2005.09.007}
	{\path{doi:10.1016/j.ejor.2005.09.007}}.
	
	\bibitem{wu2}
	{H.C. Wu}.
	\newblock {The Karush-Kuhn-Tucker optimality conditions in multiobjective
		programming problems with interval-valued objective functions}.
	\newblock {\em {European Journal of Operational Research}}, {196}:{49--60.},
	{2009}.	
	\newblock \href {https://doi.org/10.1016/j.ejor.2008.03.012}
	{\path{doi:10.1016/j.ejor.2008.03.012}}.
\end{thebibliography}



\end{document}